\documentclass[a4paper]{amsart}
\usepackage{amsmath}
\usepackage{amsthm}
\usepackage{epsfig}
\usepackage{psfrag}
\usepackage{graphicx}
\usepackage{graphpap,latexsym}
\usepackage{color}
\usepackage{ulem}
\usepackage{amssymb,mathrsfs,enumerate}

\newcommand{\R}{\Rz}

\newcommand{\Rz}{\mathbb{R}}
\newcommand{\RR}{\mathbb{R}}
\newcommand{\NN}{\mathbb{N}}
\newcommand{\RN}{\mathbb{R}^N}

\newcommand{\RdN}{\mathbb{R}^{2N}}
\newcommand{\RNxi}{\mathbb{R}^N_{\xi}}
\newcommand{\Nz}{\mathbb{N}}

\newcommand{\rhs}{right hand side}
\newcommand{\lhs}{left hand side}
\newcommand{\hbeta}{\hat\beta}

\newcommand{\epsi}{\varepsilon}

\newcommand{\s}{\sigma}

\newcommand{\vf}{\varphi}

\DeclareMathOperator{\sign}{sign}

\newcommand{\Sch}{\mathcal S(\RN)}

\newcommand{\calV}{\mathcal{V}}
\newcommand{\calH}{\mathcal{H}}

\newcommand{\ee}{_\varepsilon}

\newcommand{\dx}{{\rm d} x}
\newcommand{\dxi}{{\rm d} \xi}
\newcommand{\dz}{{\rm d} z}

\newcommand{\dy}{{\rm d} y}

\newcommand{\dt}{{\rm d} t}

\newcommand{\vep}{\varepsilon}
\renewcommand{\d}{{\rm d}}
\newcommand{\Csi}{\mathbb{C}_{\sigma}}

\newcommand{\sigmak}{\sigma_{k}}

\newcommand{\Xzs}{\mathcal{X}_{s,0}}
\newcommand{\Xzsk}{\mathcal{X}_{s_k,0}}
\newcommand{\Xzsi}{\mathcal{X}_{\sigma,0}}
\newcommand{\Xzsik}{\mathcal{X}_{\sigmak,0}}
\newcommand{\Xzr}{\mathcal{X}_{r,0}}
\newcommand{\Xzrk}{\mathcal{X}_{r_k,0}}

\newcommand{\Xzb}{\mathcal{X}_{\beta,0}}
\newcommand{\Xb}{\mathcal{X}_{\beta}}
\newcommand{\Xr}{\mathcal{X}_{r}}

\newcommand{\Asig}{\mathfrak{A}_\sigma}
\newcommand{\Asi}{\mathfrak{A}_\sigma}
\newcommand{\Asigk}{\mathfrak{A}_{\sigma_k}}
\newcommand{\As}{\mathfrak{A}_s}
\newcommand{\Ask}{\mathfrak{A}_{s_k}}
\newcommand{\Ar}{\mathfrak{A}_r}
\newcommand{\Arm}{\mathfrak{A}_{r/2}}
\newcommand{\Askk}{\mathfrak{A}_{s_k}}
\newcommand{\Ark}{\mathfrak{A}_{r_k}}

\newcommand{\B}{B}

\newtheorem{corollary}{Corollary}[section]
\newtheorem{propo}{Proposition}[section]
\newtheorem{remark}{Remark}[section]
\newtheorem{lemma}{Lemma}[section]
\newtheorem{defi}{Definition}[section]
\newtheorem{theorem}{Theorem}

\newcommand{\LR}{L^2(\RN)}
\newcommand{\LO}{H_0} 
\newcommand{\LpO}{L^p_0(\RN)}
\newcommand{\LpOd}{L^{p'}_0(\RN)}

\newcommand{\DO}{\mathscr D_{\Omega}(\RN)}

\newcommand{\lista}{\begin{list}{}{\setlength{\leftmargin}{0.6in}\setlength{\labelwidth}{1.5in}\setlength{\labelsep}{0.2in}}}
\newcommand{\finelista}{\end{list}} 

\newcommand{\Xsigz}{\mathcal{X}_{\sigma,0}}
\newcommand{\Esig}{\mathcal{E}_{\sigma}}
\newcommand{\Esik}{\mathcal{E}_{\sigma_k}}
\newcommand{\Xsz}{\mathcal{X}_{s,0}}

\newcommand{\Ds}{(-\Delta)^s}

\newcommand{\Drk}{(-\Delta)^{r_k}}
\newcommand{\Dr}{(-\Delta)^r}
\newcommand{\Drreg}{(-\Delta)^r_{\rm reg}}
\newcommand{\Dsireg}{(-\Delta)^\sigma_{\rm reg}}
\newcommand{\Dre}{(-\Delta)^r_\epsi}
\newcommand{\Kre}{K_{r,\epsi}}

\newcommand{\Dsig}{(-\Delta)^\sigma}

\newcommand{\Drkm}{(-\Delta)^{r_{k}/2}}

\newcommand{\pv}{\mathop{p.v.}}

\newcommand{\Esi}{\mathbb{E}_{\sigma}}
\newcommand{\Esk}{\tilde{\mathbb{E}}_{\sigma_k}}
\newcommand{\Ezsi}{\mathbb{E}_{\sigma}(u_0)}
\newcommand{\Ezsit}{\tilde{\mathbb{E}}_{\sigma_k}(u_{0,k})}

\newcommand{\RRR}{\color{red}}  

\newcommand{\EEE}{\color{black}}

\usepackage{version}	
\begin{document}

\title{Fractional Cahn-Hilliard, Allen-Cahn and porous medium equations}

\pagestyle{myheadings}

\author{Goro Akagi}
\address{Graduate School of System Informatics, Kobe University}
\author{Giulio Schimperna}
\address{Dipartimento di Matematica ``F. Casorati'', via Ferrata 1, I--27100 Pavia, Italy}
\author{Antonio Segatti}
\address{Dipartimento di Matematica ``F. Casorati'', via Ferrata 1, I--27100 Pavia, Italy}
\email{akagi@port.kobe-u.ac.jp}
\email{giusch04@unipv.it}
\email{antonio.segatti@unipv.it}
\urladdr{http://www2.kobe-u.ac.jp/~akagi56/index.html}
\urladdr{http://www-dimat.unipv.it/giulio}
\urladdr{http://www-dimat.unipv.it/segatti}


\begin{abstract}
 We introduce a fractional variant of the Cahn-Hilliard equation
 settled in a bounded domain $\Omega\subset \RR^N$ and complemented with 
 homogeneous Dirichlet boundary conditions of solid type
 (i.e., imposed in the whole of $\RR^N\setminus\Omega$). 
 After setting a proper functional framework,
 we prove existence and uniqueness of weak solutions 
 to the related initial-boundary value problem. 
 Then, we investigate some significant singular 
 limits obtained as the order of either of the fractional
 Laplacians appearing in the equation is let tend to $0$.
 In particular, we can rigorously prove that
 the fractional Allen-Cahn, fractional porous medium,
 and fractional fast-diffusion equations can be
 obtained in the limit. Finally, in the last part
 of the paper, we discuss existence and qualitative
 properties of stationary solutions
 of our problem and of its singular limits.
\end{abstract}

\maketitle

%
%

\section{Introduction}
\label{sec:intro}

Let $\Omega$ be a smooth bounded domain in $\mathbb{R}^N$. 
For $s, \sigma \in (0,1)$, 
we consider the following class of initial and boundary value problems:
\begin{alignat}{4}\label{eq:fCH}
  & \partial_t u  + \Ds w = 0 \ &\hbox{ in }&  \Omega \times (0,+\infty), \\
  & w = \Dsig u + W'(u) \ &\hbox{ in }& \Omega\times (0,+\infty),  \label{eq:chem_pot}\\
  & u(x,0) = u_0(x) \ &\hbox{ in }& \Omega, \label{eq:iniz_bc}\\
  & u = w = 0 \ &\hbox{ in }& \RN\setminus \Omega.\label{eq:bc}
\end{alignat}
The above system constitutes a natural generalization of the well-known 
and extensively studied \emph{Cahn-Hilliard equation},
{\em Allen-Cahn equation}, 
and {\em porous medium equation}.
More precisely, when $s=\sigma =1$, 
system \eqref{eq:fCH}-\eqref{eq:chem_pot} 
reduces to the Cahn-Hilliard equation 
(cf.~\cite{CH}), when $s=0$ and $\sigma =1$ we have 
the Allen-Cahn equation (cf.~\cite{AC}), and 
when $s=\sigma=0$, \eqref{eq:fCH}-\eqref{eq:chem_pot} 
turns to an ODE. The relations between the above system and
the porous medium equation will be outlined and 
made rigorous later on.

The function $W$ in \eqref{eq:chem_pot} represents a configuration
potential which may have two (or more) wells. The general structure
of $W$ is given by
\begin{equation}\label{general:W}
  W(v) = \hat\beta(v) -  \frac{\lambda}2 v^2,
   \quad\text{so that } W'(v) = \beta(v) - \lambda v,
   \quad\text{where } \beta = \hat\beta'.
\end{equation}
Here $\hat\beta$ is a smooth and convex function and 
$\lambda \ge 0$ is a constant. Hence, if $\lambda>0$,
then $W$ may be nonconvex. In the phase-transition literature,
the wells of $W$ correspond to energy minima (attained at 
pure phases or configurations). In view of the variational structure
of our system, it is convenient to keep the same interpretation
also in the present case. Actually, several types of
significant choices have been proposed for $W$, also including
cases where $\hat\beta$ is nonsmooth or even singular 
(like the so-called {\it logarithmic potential}\/
$\hat\beta(v)=(1-v)\log(1-v)+(1+v)\log(1+v)$,
$v\in(-1,1)$).
For the sake of simplicity, in this paper we will 
just consider the case given by
\begin{equation}\label{choice:pot}
  \hat\beta(v)=\frac1p |v|^p, \ \ p\in(1,\infty),  
   \quad\text{so that } \beta(v) = |v|^{p-1}\sign v.
\end{equation}
Moreover, we will set $\lambda=1$.
This choice, up to an additive constant, includes the standard
double-well potential, namely $W(v)=\frac14 (v^2-1)^2$,
widely used in the literature. 
Mathematically speaking, the cases $p>2$ and $p<2$ enjoy
rather different features. Indeed, while in the former case 
$W(v)$ is coercive as $|v|\nearrow \infty$,
in the latter situation $W$ is unbounded from below. 
Hence, we can expect (and, indeed, we will prove) 
different behaviors of the solutions, especially
for large values of the time variable.
Note that the case $p=2$ corresponds in fact to the 
linear problem and is not considered here.

A further important feature of our
problem is the occurrence of {\it solid}\/ boundary 
conditions of homogeneous Dirichlet type,
stated by \eqref{eq:bc}. Namely, the values of $u$ and $w$ 
are prescribed in the {\it whole complement}\/
of $\Omega$, not only on the boundary. Of course, this 
assumption is strongly related to the {\it nonlocal}\/ 
character of fractional Laplacians.
Indeed, it is worth observing from the very beginning that, 
though the values of $(-\Delta)^s w$ and of $(-\Delta)^\sigma u$
at any point $x\in \Omega$ depend also on the values of $u$ and $w$ 
outside $\Omega$ (which are set to be $0$ by the boundary conditions),
we prescribe the validity of \eqref{eq:fCH}
and \eqref{eq:chem_pot} only at the points $x \in \Omega$. For 
$x \in \RR^N\setminus \Omega$,
\eqref{eq:fCH}-\eqref{eq:chem_pot}
need not to be satisfied (and, indeed, there is no 
reason why they should). This observation will be further 
clarified in Section~\ref{sec:not}, where the appropriate 
concept of weak solution is introduced. Correspondingly, 
we will also recall the precise definition of 
$(-\Delta)^r$, $r\in(0,1)$, both in the strong and 
in the weak (variational) form, the latter being 
the more appropriate one for the analysis of our problem.

\smallskip
The study of system \eqref{eq:fCH}-\eqref{eq:bc}
is motivated both from the point of applications and due to 
its mere mathematical interest. Under the first perspective, 
it is worth recalling that the Allen-Cahn and Cahn-Hilliard 
equations, in their standard formulation (i.e., with
the usual Laplace operators), play a central
role in materials science. Indeed, they commonly
occur in mathematical models for 
phase-transition or separation,
viscoelasticity, damaging, complex fluids, and 
whenever {\em diffuse interfaces}\/ appear
(see, e.g., \cite{CMZ,N-C} 
for a comprehensive bibliography).
A reason for considering a fractional version 
of the Cahn-Hilliard equation can be provided 
by observing that, in the original 
formulation of the physical model \cite{CH}, 
the Laplace operator in \eqref{eq:chem_pot} was
actually replaced by a spatial convolution 
term, aimed at describing long-range interactions
among particles. It was only in the subsequent
mathematical literature that, mainly for analytical
reasons, this nonlocal term has been substituted with the 
term $- \Delta u$. Under this perspective,
the use of the fractional Laplacian 
$(-\Delta)^\sigma$ (which, at least for smooth functions, 
may be represented exactly by a convolution integral) 
appears to be more adherent to the physical setting. 
It is worth remarking that the study of fractional 
(or, more generally, nonlocal) PDE's 
is a lively research topic, both from the point
of view of mathematical theory and 
in relation with the many real-world applications.
Among these, we mention obstacle problems \cite{caffa_salsa_silve},
finance \cite{cont}, quasi-geostrophic flows 
\cite{caffa_vasse10,costa_wu99}, 
anomalous diffusion \cite{MK,Wo,Za}. 
A more comprehensive list of references is provided 
in the survey~\cite{Dine_Pala_Vald}.

In the recent literature, a relevant number of works have 
been devoted to the analysis of nonlocal Allen-Cahn and 
Cahn-Hilliard models. Here we quote, among others, 
\cite{BatesHan,GZ,NNV} (see also \cite{CFG} for an application 
to complex fluids). Actually, in \cite{BatesHan,CFG,GZ}, 
the term $\Dsig u$ in \eqref{eq:chem_pot} is replaced by 
\begin{equation}\label{eq:bates}
  J[u] = a (\cdot) u - j * u, 
   \quad a(x) = \int_\Omega j (x - y) \, \dy, 
   \quad (j * u) (x) = \int_\Omega j (x - y) u(y) \, \dy, 
\end{equation}
where the convolution kernel $j$ enjoys suitable, 
and rather strong, regularity properties. For instance,
in \cite{BatesHan}, $j$ is assumed to lie in $C^{2+\alpha}(\RR^N)$ 
for some $\alpha>0$, while in \cite{CFG} $j$ is taken 
in $W^{1,1}(\RR^N)$, and in \cite{GZ} a similar 
condition is required. In all cases, 
both integrals in the \rhs\ of \eqref{eq:bates} 
are finite for (almost) every $x\in \Omega$, whenever 
$u$ lies, say, in $L^1(\Omega)$ (more comments on this point
will be given in the Appendix). 

On the other hand, the kernel $K_r$ generating the 
fractional Laplacian $(-\Delta)^r$ 
(cf.~\eqref{def:fract_lapl} below) 
is not even summable. Hence, the corresponding 
convolution integral needs
to be intended in the principal value sense
even for smooth functions $u$ (cf.~\eqref{p.v.}).
In addition to that, if $u$ is not regular (for instance if
it just lies in some $L^p$-space), then the ``pointwise'' 
expression \eqref{def:fract_lapl} makes no sense at all and 
a {\em variational}\/ definition of $(-\Delta)^r u$
is required. In this sense, the fractional Laplacian 
gives rise to a much stronger singularity with respect to
those considered in \cite{BatesHan,CFG,GZ}.
Up to our knowledge, the only papers dealing 
with a true fractional 
Allen-Cahn model are \cite{NNV} and 
\cite{DiPaVa13}, where, however, the analysis is 
restricted to the spatial one-dimensional case 
and mostly concentrated on other aspects rather than 
weak solvability and regularity of solutions.

More recently, Abels, Bosia and Grasselli in \cite{ABG} analyzed
a variant of problem \eqref{eq:fCH}-\eqref{eq:chem_pot},
where the diffusion operator in (the analogue of)
\eqref{eq:fCH} is the standard Laplacian, 
while the diffusion operator in (the analogue of)
\eqref{eq:chem_pot} is the so-called {\it regional}\/
fractional Laplacian. For its definition we refer
the reader to the Appendix; here it is just worth mentioning
that its properties are slightly different compared to 
those of the operator considered in this paper, 
the main point regarding the boundary conditions. Actually,
the operator of \cite{ABG} can be seen as a fractional
power of the Neumann Laplacian on $\Omega$ (in particular,
this implies conservation of mass, which does not hold
here in view of the Dirichlet condition~\eqref{eq:bc}).
The authors of \cite{ABG} prove existence and uniqueness 
of weak solutions in the case when the function $W$ may 
have a singular character (as happens for the logarithmic 
potential mentioned before). Moreover they characterize
the long-time behavior of solution trajectories proving 
existence of the global attractor for the dynamical
process associated to the system. 

\smallskip
The first aim of this paper is to prove existence
and uniqueness of solutions to a weak formulation of 
problem~\eqref{eq:fCH}-\eqref{eq:bc}. 
Hence, differently from \cite{ABG}, we will
consider a fractional dynamics both in~\eqref{eq:fCH} 
and in~\eqref{eq:chem_pot}. This 
program requires, at first, to set a proper functional 
framework. Our approach basically follows (and complements)
the perspective given in \cite{serva-valdi11}, where 
a weak version of the fractional operators ruling the system
is defined. Actually, a variational expression of the Dirichlet
fractional Laplacian can be given at least in two (equivalent)
%
%
ways: one can either work with functions defined 
on $\Omega$ and implicitly extend them to $0$ outside
 $\Omega$ when computing fractional
Laplacians (which depend on values taken on the whole of~$\RR^N$),
or one may use spaces of functions defined on the whole space
but constrained to be identically equal to zero outside $\Omega$.
Generally, we shall work within the latter framework.
This choice permits us to address the problem, 
for all fixed $s,\sigma \in (0,1)$, 
in the usual Hilbert setting, very similarly 
to what happens for the standard Cahn-Hilliard
model. In particular, we can prove existence by means of 
a classical time-discretization scheme.
Compactness and duality arguments are then exploited
in order to pass to the limit in the discretization;
in this way we can avoid any reference to 
finer regularity properties of solutions to fractional 
elliptic and parabolic problems, which may involve 
rather delicate issues. Uniqueness also follows
from a simple contraction principle.

\smallskip
As anticipated, after establishing well-posedness of the model, we
will turn our attention to further properties of solutions.
As a first issue, we will let the ``order'' $\sigma$ of 
the fractional Laplacian in \eqref{eq:chem_pot} go to $0$.
This singular limit is motivated by noting that, as
$\sigma \searrow 0$, $(-\Delta)^\sigma u \to u$ in a suitable sense.
As a consequence we obtain, at least formally,
$$
  w = (-\Delta)^\sigma u + W'(u) 
   = (-\Delta)^\sigma u + \beta(u) - u 
   \stackrel{\sigma\searrow 0}{\longrightarrow} \beta(u).
$$
Hence, on account of \eqref{eq:fCH}
and recalling assumption \eqref{choice:pot}, one
expects to get in the limit the equation
\begin{align}\label{eq:fract_pm}
  \partial_t u  
    + \Ds \big( |u|^{p-1} \sign u \big) & = 0 
          \ \hbox{ in }  \Omega \times (0,+\infty),\\
  u & = 0 \ \hbox{ in } \RN\setminus \Omega.
\end{align}
This corresponds, for $p> 2$, to the so-called 
\emph{fractional porous medium equation},
recently addressed and studied in a number of contributions 
(see, among others, \cite{vazquez_fractpm1}, \cite{Bonf_Vaz12}, 
\cite{Bonf_Vaz13} and \cite{Bonf_Sire_Vaz14}). 

Actually, taking a family $\{(u_\sigma,w_\sigma)\}$ 
of solutions to our problem, and letting
$\sigma\searrow 0$, we can rigorously prove that,
up to extraction of a subsequence, $u_\sigma$ tends
to a limit function $u$ satisfying~\eqref{eq:fract_pm}. 
Our result holds, under natural assumptions on 
the initial data, both for $p>2$ and for $p\in (1,2)$.
The proof is, however, not straightforward
and relies on some fine properties of first eigenvalues 
of fractional elliptic Dirichlet problems, which, 
to the best of our knowledge, are new (see Proposition \ref{prop:asy_1eigen}).
The case $p\in (1,2)$ is a bit more involved due to
the lack of coercivity of the energy functional. 
Indeed, in that case we need to modify a bit the
energy (see Theorem \ref{th:fd} below) in such a way
to get an estimate for 
$u_\sigma$ uniform as $\sigma\searrow 0$. 
However, this modification
does not affect the limit equation~\eqref{eq:fract_pm},
which corresponds in this case to
the \emph{fractional fast-diffusion equation} studied, 
e.g., in \cite{kim_lee}.

On the other hand, it is also natural to investigate what
happens as one lets $s\searrow 0$. In that case, 
one expects that
$$ 
  \Ds w \stackrel{s\searrow 0}{\longrightarrow} w.
$$
In other words, the limit gives rise to 
the \emph{fractional Allen-Cahn equation}
\begin{align*}
  \partial_t u + \Dsig u + W'(u) & = 0 \ \hbox{ in } \Omega\times (0,+\infty),\\
  u & = 0 \ \hbox{ in } \RN \setminus \Omega,
\end{align*}
both for $p > 2$ and for $1 < p < 2$. 
It is worth noting that the limit $s\searrow 0$ 
is considerably simpler than the limit $\sigma \searrow 0$
since keeping $\sigma$ fixed maintains some additional 
space compactness, which reveals to be helpful for the purposes of
obtaining a strong convergence for $u$
and identifying the nonlinear terms in the limit.
As a result of these two procedures, we can see problem
\eqref{eq:fCH}-\eqref{eq:bc} as a bridge 
between the usual Cahn-Hilliard equation (given by $(\sigma,s)=(1,1)$)
and the fractional (or also non-fractional) porous medium 
and Allen-Cahn equations. 

\smallskip
The last part of the paper is devoted to the proof of
some results related to stationary solutions to problem
\eqref{eq:fCH}-\eqref{eq:bc}
(for {\it fixed}\/ $s,\sigma \in (0,1)$).
As expected in view of the nonconvex character 
of $W$, we can show that nontrivial
stationary states exist if and only if the first eigenvalue 
$\lambda_1(\sigma)$ of $(-\Delta)^\sigma$ is 
strictly smaller than $1$. Indeed, whenever 
$\lambda_1(\sigma)\ge 1$, the coercivity
given by $(-\Delta)^\sigma$ compensates the nonconvexity
of $W'(u) = \beta(u) - u$, exactly as happens for the 
standard Laplacian. Of course, the properties of 
the stationary states play an important role for what
concerns the long-time behavior of solutions to the
evolutionary system. We plan to address this issue,
and, particularly, to investigate the properties of 
{\it $\omega$-limit sets}, in a forthcoming paper.

%
%
%
%

\smallskip
The plan of the paper is as follows. A survey on some basic definitions
and tools related to the fractional Laplacian, as well as some useful
lemmas, are presented in the next Section~\ref{sec:not}. In
Section~\ref{sec:main}, we introduce our assumptions
and state our main results. The proofs of existence and regularity
properties of solutions are carried out in Section~\ref{sec:conti}, 
while the convergence to the fractional porous medium and Allen-Cahn
equations is analyzed in Section~\ref{sec:conve}.
Our last results regarding the properties of stationary states 
are presented in~Section~\ref{sec:stat}. Finally, in the Appendix we
provide~some more comments on the relations occurring 
between the problem analyzed here and other nonlocal models of 
Allen-Cahn or Cahn-Hilliard type studied in the literature.
%
%

%
%

\section{Notations and Preliminaries}
\label{sec:not}

\subsection{The fractional Laplacian}%
We introduce here the standard (strong) 
form of the fractional Laplace operator
$\Dr$ in the whole space $\RR^N$. The reader may refer, e.g., to 
\cite{Dine_Pala_Vald} for additional details.
Given $r\in (0,1)$, for $u$ in the Schwartz
class $\Sch$ of the rapidly decaying functions at infinity,
$\Dr u$ is defined as 
\begin{equation}\label{def:fract_lapl}
  \Dr u(x):= C(r,N) \pv \int_{\RN} \frac{u(x)- u(y)}{\vert x-y\vert^{N + 2 r}} \,\dy,
\end{equation}
where the notation $\pv$ means that the integral is taken in the 
{\em Cauchy principal value}\/ sense, namely
\begin{equation}\label{p.v.}
  \pv \int_{\RN} \frac{u(x)- u(y)}{\vert x-y\vert^{N + 2 r}} \,\dy
   = \lim_{\epsilon\searrow 0} \int_{\RN\setminus B(x,\epsilon)} 
          \frac{u(x)- u(y)}{\vert x-y\vert^{N + 2 r}} \,\dy.
\end{equation}
The exact value and the asymptotics with respect to $r$
of the normalizing constant $C(r,N)$ are crucial
for our purposes. To this end, we recall that $C(r,N)=(\int_{\RN}\frac{1-\cos(\zeta_1)}{\vert\zeta\vert^{N+2r}}\hbox{d$\zeta$})^{-1}$
and that (see, e.g., \cite[Corollary 4.2]{Dine_Pala_Vald})
\begin{equation}
\label{eq:limC}
  \lim_{r\searrow 0}\frac{C(r,N)}{r(1-r)} = \frac{2}{\vert {\mathbb S}^{N-1}\vert}.
\end{equation}
On the other hand, since the dependence of $C(r,N)$ with respect to 
the space dimension $N$ is not the major issue for this paper, 
we will always write $C(r)$ for $C(r,N)$ in what follows.
For any $r\in (0,1)$ and for any $x,y\in \RN$ we 
will also use the shorthand notation 
$$
K_r(x-y)=\vert x-y\vert^{-N-2r} 
$$
to denote the singular kernel appearing in the definition
of $\Dr$. A second, albeit equivalent, definition can be given using 
the Fourier transform. Indeed, $\Dr$ can be introduced as
the pseudo-differential operator of symbol $\vert \xi\vert ^{2r}$,
namely
\begin{equation}
\label{def:fract_laplbis}
  \Dr v = \mathfrak{F}^{-1}(\vert \xi\vert ^{2r}\mathfrak{F}(v)),
   \quad \forall\, v\in \Sch.
\end{equation}
We denote by $\mathfrak{F}(v)$ (or by $\hat{v}$) the Fourier transform 
of $v$. 
\subsection{Fractional Sobolev spaces}%

In this subsection, we shall deal with \emph{fractional Sobolev spaces}.
We refer the reader to, e.g.,~\cite{lions_mag} and~\cite{Adams} for
further details. 

For $r \in \R$, the fractional Sobolev space $H^r(\RN)$ is defined by
$$
H^r(\RN) := \left\{ v \in \Sch' \colon (1 + |\xi|^2)^{r/2}
\hat v(\xi) \in L^2 (\RN_\xi) \right\},
$$
where $\Sch'$ stands for the dual space of the Schwartz class
$\Sch$ and $L^2(\RN_\xi)$ is the space of square-integrable
functions with respect to the variables $\xi \in \mathbb R^N_\xi$, equipped
with the norm
$$
\|v\| := \left\| (1 + |\xi|^2)^{r/2}
\hat v(\xi) \right\|_{L^2(\RN_\xi)}
\quad \mbox{ for } \ v \in H^r(\RN).
$$
In particular, for $r \in (0,1)$, we can equivalently write 
$$
H^r(\RN) := \left\{
v \in L^2(\RN) \colon (x,y) \mapsto
K_r(|x-y|)|v(x)-v(y)|^2 \in L^1(\RdN)
\right\},
$$
endowed with the (equivalent) norm 
\begin{equation}\label{eq:normHsig}
  \|v\|^2_{H^r(\RN)}:= \| v\|^2_{L^2(\RN)} +
   \frac{C(r)}{2}[v]^2_{H^r(\RN)}
\quad \mbox{ for } \ v \in H^r(\RN).
\end{equation}
Here, $[\,\cdot\,]_{H^r}$ denotes the so-called {\em Gagliardo-seminorm}
$$
  [v]^2_{H^r(\RN)}:= \iint_{\RdN} K_r(x-y)\vert v(x)-v(y) \vert^2 \,\dx \, \dy.
$$
Furthermore, for $r \in \R$,  $H^r(\RN)$ is an intermediate space
 between $H^m(\RN)$ and $L^2(\RN)$, that is,
$$
H^r(\RN) = \left[ H^m(\RN), L^2(\RN) \right]_\theta
$$
for any $\theta \in (0,1)$ and $m \in \mathbb Z$ satisfying
$r = (1-\theta) m$.

Let $\Omega$ be a bounded domain of $\RN$ with smooth boundary $\partial
\Omega$. For $r \in \R$, the fractional Sobolev space $H^r(\Omega)$ may be analogously 
defined as
$$
H^r(\Omega) := \left[ H^m(\Omega), L^2(\Omega) \right]_{\theta}
$$
for any $\theta \in (0,1)$ and $m \in \mathbb N$ satisfying $(1-\theta)m
= r$.  
Moreover, for $r \in (0,1)$, one may use an alternative definition,
$$
H^r(\Omega) := \left\{
u \in L^2(\Omega) \colon (x,y) \mapsto K_r(|x-y|)|u(x)-u(y)|^2 \in
L^1(\Omega \times \Omega)
\right\}
$$
with the intrinsic norm
$$
\|v\|_{H^r(\Omega)}^{2} := \|v\|_{L^2(\Omega)}^{2} + \dfrac{C(r)}2 
\iint_{\Omega \times \Omega} K_r(|x-y|) |v(x)-v(y)|^2 \; \d x \, \d y
\quad \mbox{ for } \ v \in H^r(\Omega).
$$
Sobolev embeddings and inequalities also hold for fractional
Sobolev spaces. Thus, 
$H^r(\Omega)$ is continuously (resp., compactly)
embedded in $L^q(\Omega)$, provided that $1 \leq q \leq 2_r^* :=
2N/(N-2r)$ (resp., $1 \leq q < 2_r^*$) and $2r < N$.
Finally, for each $r > 0$, $H^r_0(\Omega)$ is defined as the closure of
$C^\infty_0(\Omega)$ in $H^r(\Omega)$. In case $r \leq 1/2$, the space
$H^r_0(\Omega)$ coincides with $H^r(\Omega)$; in case $r > 1/2$,
$H^r_0(\Omega)$ is strictly contained in $H^r(\Omega)$ 
(see, e.g.,~\cite[p.~55, Theorem 11.1]{lions_mag}).


%
\subsection{The functional framework}\label{Ss:FF}%
It is apparent from \eqref{def:fract_laplbis} that, for $v \in \Sch$,
$\Dr v$ does not necessarily belong to $\Sch$ (being $r<1$, 
the symbol $\vert \xi\vert^{2r}$ introduces
a singularity in the origin in its Fourier transform). Moreover,
even for $v$ with compact support, $\Dr v$ generally does
not have compact support due to the non locality of the operator.
In addition to this, the above definition could make no sense
when non-smooth functions are involved. 
Thus, it will be important for us to extend the definition of the
fractional Laplacian to a more general setting. 
This will be accomplished by 
using the theory of distributions together with some tools of convex
analysis. The framework we are going to fix will permit us 
to use variational and energy techniques 
in order to address our problem. 
As already observed in the introduction, although equations
\eqref{eq:fCH}-\eqref{eq:chem_pot}
are settled only in $\Omega$, the behavior of $\Dsig u$ and 
$\Ds w$ depends on the interplay between the values of $u$
and of $w$ inside and outside $\Omega$. 
Proceeding along the lines of \cite{serva-valdi11},
we can then introduce some functional spaces. 

Firstly, we set
\begin{equation}\label{defiH}
   \LO:=\big\{ v \in L^2(\RN): v=0~\text{a.e.~in $\RR^N\setminus \Omega$} \big\}.  
\end{equation}
The space $\LR$ (hence its closed subspace $\LO$)
is endowed with its standard scalar product,
$$
(u,v) := \int_{\RN} u(x)v(x) \, \d x \quad \mbox{ for } \ u,v \in \LR.
$$
Of course, in the closed subspace $\LO$ 
taking the scalar product of $L^2(\Omega)$ and the associated norm
would make no difference.
%
%
Hence, $L^2(\Omega)$ can be identified with $\LO$ by zero
extension outside $\Omega$. Furthermore, we set 
\begin{align*}
\LpO &:= \left\{
v \in L^p(\RN) \colon v = 0 \ \mbox{ a.e.~in } \RN \setminus \Omega 
\right\} \ \mbox{ with } \ \|\cdot\|_{\LpO} := \|\cdot\|_{L^p(\RN)},\\
\DO &:= \left\{
\phi \in C^\infty(\RN) \colon \phi|_{\Omega} \in C^\infty_0(\Omega) \
\mbox{ and } \ \phi = 0 \ \mbox{ in } \RN \setminus \Omega
\right\},
\end{align*}
which can be identified with $L^p(\Omega)$ and $C^\infty_0(\Omega)$,
respectively. Then $\LpO$ is reflexive for $p \in (1,\infty)$, since
$\LpO$ is closed in $L^p(\RN)$; moreover, $\LpO$ is separable for $p \in
[1,\infty)$. Furthermore, $\DO$ is dense in $\LpO$, provided that $1 \leq
p < \infty$. 

As in \cite{serva-valdi11}, we set
$Q:= \RdN \setminus\big((\RN\setminus\Omega)\times (\RN\setminus\Omega) \big)$
and denote by $\Xzr$, $r\in (0,1)$, the space
\begin{equation}\label{eq:defX0}
  \Xzr:=\left\{v\in \LO \colon (x,y)\mapsto (v(x)-v(y))\sqrt{K_r(x-y)}\in L^2(Q) \right\}.
\end{equation}
Actually, $\Xzr$ can be endowed with the scalar product
\begin{equation}\label{eq:proscX0}
 (v,z)_{\Xzr} := ( v, z) 
   + \frac{C(r)}{2} \iint_{Q} K_r(x-y) (v(x)-v(y)) (z(x)-z(y))
   \, \dx \, \dy
\quad \mbox{ for } \ v,z \in \Xzr 
\end{equation}
and the associated norm
\begin{equation}\label{eq:normX0}
 \| v\|^2 := \|v\|^2_{L^2(\Omega)} 
   + \frac{C(r)}{2} \iint_{Q} K_r(x-y)\vert v(x)-v(y) \vert^2 \, \dx \, \dy
\quad \mbox{ for } \ v \in \Xzr, 
\end{equation}
where $C(r)$ is as in \eqref{def:fract_lapl}. 
Then, it is easy to check that $\Xzr$ is a Hilbert space (i.e.,
the above norm is complete). Note that $\Xzr$ could be 
also presented in a more familiar form,
\begin{equation}\label{eq:defX0_bis}
  \Xzr = \big\{ v\in H^r(\RN) \hbox{ such that } 
     v=0 \hbox{ a.e.~in } \RN\setminus \Omega \big\}.
\end{equation}

There holds the Poincar\'e-type inequality,
\begin{equation}\label{eq:poincare}
  \| v\|^2_{L^2(\Omega)}\le c_{P}(r) \frac{C(r)}{2}[v]^2_{H^r(\mathbb{R}^N)}
\quad \mbox{ for all } \ v\in \Xzr
\end{equation}
for some constant $c_{P}(r)$ depending only on $r$, $N$ and the diameter
of $\Omega$. Indeed, one can take $R > 0$ such that $\Omega$ is included
in the open ball $B_R$ of radius $R$ centered at the origin. 
Then, using the definition of the Gagliardo-seminorm, 
we see that
\begin{align*}
 [v]_{H^r(\mathbb{R}^N)}^2 &\geq \int_{\Omega^c}\int_\Omega
 K_r(|x-y|)v(x)^2 \; \d x \, \d y\\
&\geq \int_\Omega \left( \int_{\Omega^c \cap B_{R+1}} K_r(|x-y|) \; \d y
 \right) v(x)^2 \; \d x
\geq \dfrac{|B_{R+1} \setminus \Omega|}{(2R+2)^{N+2r}} \|v\|_{L^2(\Omega)}^2
\quad \mbox{ for all } \ v \in \Xzr,
\end{align*}
where $\Omega^c$ stands for the complement of $\Omega$ and $|B_{R+1}
\setminus \Omega|$ denotes the Lebesgue measure of the set $B_{R+1}
\setminus \Omega$. Note that $|B_{R+1} \setminus \Omega| \geq |B_{R+1}
\setminus B_R| > 0$. Thus the quotient
$[v]_{H^r(\RN)}^2/\|v\|_{L^2(\Omega)}^2$ is bounded from below for $v \in
\Xzr$, whence \eqref{eq:poincare} follows.
Hence, by \eqref{eq:poincare}, the norm on $\Xzr$ given by
\begin{equation}\label{eq:normX0bis}
  \| v \|^2_{\Xzr} := 
    \frac{C(r)}{2} \iint_{Q} K_r(x-y) \vert v(x)-v(y) \vert^2
    \, \dx\, \dy
\quad \mbox{ for } \ v \in \Xzr 
\end{equation}
is equivalent to that defined in \eqref{eq:normX0}. From
now on, we will fix \eqref{eq:normX0bis} to be the norm 
in $\Xzr$.

%
Now let us introduce the dual spaces $\LO'$ and $\Xzr'$ of $\LO$ and
$\Xzr$, respectively. In particular, we are here concerned with the
meaning of the \emph{equality in the dual space},
$$
f = g \ \mbox{ in } \LO' \quad \mbox{ for } \ f, g \in \LO',
$$
which actually means 
\begin{equation}\label{equiv-dual}
\langle f, v\rangle_{\LO} = \langle g, v\rangle_{\LO} \ \mbox{ for
all } v \in \LO.
\end{equation}
By the Riesz representation theorem, one can uniquely take $u_f, u_g \in
\LO$ such that $\langle f , v \rangle_{\LO} = (u_f, v)$ and $\langle g , v
\rangle_{\LO} = (u_g, v)$ for all $v \in \LO$. Hence \eqref{equiv-dual}
yields $u_f = u_g$ in $\Omega$ (hence, over $\RN$). On the other hand,
one may generate $h \in \LO'$ from a function $w_h \in L^2(\RN)$ which
might not vanish outside $\Omega$ by setting 
\begin{equation}\label{pnt_repre}
\langle h, v \rangle_{\LO} := \int_{\RN} w_h(x) v(x) \, \d x
\quad \mbox{ for } v \in \LO. 
\end{equation}
Then $h$ coincides with $f$ in $\LO'$, provided that $w_h = u_f$ in
$\Omega$. In other words, even if $f$ and
$g$ have pointwise representations $w_f, w_g \in \LR$,
respectively, as in~\eqref{pnt_repre}, the relation $f = g$ in $\LO'$
ensures that $w_f(x) = w_g(x)$ for a.e.~$x \in \Omega$ only, and it does
not guarantee the coincidence of $w_f$ and $w_g$ outside $\Omega$. This
observation is also extended to the relation $f = g$ in $\Xzr'$.

From now on, we identify the Hilbert space $\LO$ with its dual
space $\LO'$ by means of the scalar product $(\cdot,\cdot)$.
More precisely, we shall identify $f \in \LO'$ with its unique
representation $u_f \in \LO$ by the Riesz representation theorem. Hence
we shall write $f = g$ \underline{in $\LO$} for claiming the equality of
$f,g \in \LO'$ (i.e., \eqref{equiv-dual}) as well. Here we should
emphasize again that whenever are given 
representatives 
$w_f$ and $w_g$ of $f, g\in H_0'$, respectively,
the identification implies that $w_f = w_g$ only in $\Omega$.
Then since $\Xzr$ can be seen as a {\it dense}\/
subspace of $\LO$, one may consider the Hilbert triple,
\begin{equation}\label{H-tr}
\Xzr \hookrightarrow \LO \simeq \LO' \hookrightarrow \Xzr',
\end{equation}
with compact and densely defined canonical injections. 
This relation will play a crucial role throughout this paper.

On the other hand, for $r\in (0,1)$, the extension
operator of $u\in \Xzr$ to $0$ outside $\Omega$
is a continuous mapping of $H^r(\Omega)\to H^r(\RN)$.
In particular (see \cite[Theorem~11.4, Chapter 1]{lions_mag}),
if $r \in (1/2,1)$, the functions in $\Xzr$ are
equal to zero, in the sense of traces, on $\partial \Omega$.
Hence, $\Xzr$ can be identified with $H^r_0(\Omega)$ in
that case, whereas $\Xzr \simeq H^r_0(\Omega) = H^r(\Omega)$
for $r \in (0,1/2)$. Finally, in the limit case 
$r=1/2$, it turns out that 
$\mathcal{X}_{1/2,0} \simeq  H_{00}^{1/2}(\Omega)$
(again, see \cite{lions_mag} for more details).

Based on this functional framework,
we can introduce, for $r\in(0,1)$, the weak form $\Ar$
of the fractional Laplacian $\Dr$. More precisely, the operator
$\Ar : \Xzr \to \Xzr'$ is defined by
\begin{equation}\label{eq:distr_lapl}
  \langle \Ar v, \phi\rangle 
    := \frac{C(r)}{2}\iint_{\RdN} K_r(x-y)(v(x)-v(y))(\phi(x)-\phi(y)) \,\dx \, \dy, 
    \ \ \text{for all } v,\phi\in \Xzr,
\end{equation}
where the integral over $\RdN$ can be equivalently replaced with an
integral over $Q$. Note that, as soon as $v,\phi\in \Xzr$, the 
integral in the right hand side is finite.
Note also that \eqref{eq:distr_lapl} can be understood as an 
integration by parts formula, at least when $v,\phi$ are 
sufficiently regular. Indeed, to see this, we define
\begin{equation}\label{flepsi}
  \Dre u(x):= C(r) \int_{\RN} \Kre(x-y) ( u(x)- u(y) ) \,\dy,
\end{equation}
where $\Kre := K_r (1 - \chi_{B(0,\epsi)})$
and $\chi_{B(0,\epsi)}$ denotes the characteristic function
of the ball $B(0,\epsi)$ in $\mathbb R^N$. Then, by symmetry
of $\Kre$,
\begin{align}\label{eq:intparts1}
  & \int_{\RN} \Dre v(x)\phi(x)\,\dx 
    = C(r) \iint_{\RdN} \Kre(x-y)(v(x)-v(y))\phi(x) \,\dx \, \dy \\
 \nonumber
  & \mbox{}~~~~~~~~
   = C(r) \iint_{\RdN} \Kre(z)(v(x)-v(x+z))\phi(x) \,\dx \, \dz \\
 \nonumber
  & \mbox{}~~~~~~~~
   = C(r) \iint_{\RdN} \Kre(z)(v(x)-v(x-z))\phi(x) \,\dx \, \dz \\
 \nonumber
  & \mbox{}~~~~~~~~
   = \frac{C(r)}2 \iint_{\RdN} \Kre(z)(v(x)-v(x-z))(\phi(x)-\phi(x-z)) \,\dx \, \dz\\ 
 \nonumber
  & \mbox{}~~~~~~~~ 
   = \frac{C(r)}2 \iint_{\RdN} \Kre(x-y)(v(x)-v(y))(\phi(x)-\phi(y)) \,\dx \, \dy.
\end{align}
Now, letting $\epsi\searrow 0$, the \rhs\ converges to $\langle \Ar v,
\phi\rangle$, provided that $v,\phi$ lie in $\Xzr$ (and, hence, a
fortiori if $v,\phi$ are smooth functions). On the other hand, whenever
$v$ is so smooth that $(-\Delta)^r v$ (i.e., the ``strong'' fractional
Laplacian of $v$ defined in \eqref{def:fract_lapl}) is represented by, 
say, an $L^2$-function, then the \lhs\ of \eqref{eq:intparts1}
converges to $\int_{\RN} \Dr v(x)\phi(x)\,\dx$. Hence, $\Ar$ 
can indeed be seen as an extension of $\Dr$ to less regular function.
Moreover, formula \eqref{eq:distr_lapl} can be also expressed
in terms of Fourier transform. Actually, for $v,\phi\in \Xzr$,
thanks also to Fubini's theorem, we have
\begin{align}\label{eq:distr_lapl2}
 \langle \Ar v, \phi\rangle 
  &= \frac{C(r)}{2}\int_{\mathbb{R}^N}\left(\int_{\mathbb{R}^N}K_r(z)
       (v(y+z)-v(y))(\phi(y+z)-\phi(y))\,\dy\right)\,\dz \nonumber\\
  &= \frac{C(r)}{2}\int_{\mathbb{R}^N}\hat{v}(\xi)
 \overline{\hat{\phi}(\xi)} \left(\int_{\mathbb{R}^N}
    K_r(z) \left|e^{i \xi\cdot z}-1\right|^2 \hbox{d}z\right)\,\dxi \nonumber\\
  & = \int_{\mathbb{R}^N}\vert\xi\vert^{r}\hat{v}(\xi)\,
 \overline{\vert\xi\vert^r \hat{\phi}(\xi)} \,\dxi,
\end{align}
since $\frac{C(r)}{2} \int_{\mathbb{R}^N}
K_r(z) |e^{i \xi\cdot z}-1|^2\, \dz  = \vert\xi\vert^{2r}$ (see
\cite{Dine_Pala_Vald}).
In particular, using \eqref{def:fract_laplbis}, we get
\begin{equation} \label{eq:distr_lapl3}
  \langle \Ar v, \phi\rangle 
    = \int_{\mathbb{R}^N}\vert\xi\vert^{r}\hat{v}(\xi)\,
    \overline{\vert\xi\vert^r \hat{\phi}(\xi)} \, \dxi
    = \int_{\mathbb{R}^N}(-\Delta)^{r/2}v(x) \, (-\Delta)^{r/2}\phi(x) \, \dx 
    = \langle \Ar\phi, v\rangle,
\end{equation}
which could serve as an alternative, albeit equivalent, definition of
the weak fractional Laplacian $\Ar$.

%
%
%
%

\smallskip

A weak form of the fractional Laplacian can be introduced
also for the whole space case, $\Omega = \RN$.
Indeed, for $r\in(0,1)$, we can set (cf.~ \eqref{eq:defX0})
\begin{equation}\label{eq:defX}
  \Xr:=\left\{v\in L^2(\RN): 
(x,y)\mapsto (v(x)-v(y))\sqrt{K_r(x-y)}\in L^2(\RdN) \right\}
= H^r(\RN). 
\end{equation}
Then, the previous discussion extends, with minor modifications,
to the space $\Xr$. Moreover, one can correspondingly consider
the Hilbert triplet $(\Xr,H,\Xr')$. With a small abuse of
notation we will 
indicate with the same symbol $\Ar$ the weak form of the fractional
Laplacian as an operator from $\Xr$ to $\Xr'$ defined by
\eqref{eq:distr_lapl} with $\Xzr$ replaced by $\Xr$.

The relations between the Gagliardo-seminorm and the Fourier-transform
definition of the fractional Laplacian are also clarified by the following
property (cf.~\cite[Propositions 3.4 \& 3.6]{Dine_Pala_Vald}):
\begin{equation}\label{eq:semi_norm}
  \frac{C(r)}{2}[v]^2_{H^r(\RN)} 
    = \| \Arm v\|_{L^2(\RN)}^2 
    =  \big\| \vert\xi\vert^r \hat{v} \big\|^2_{L^2(\RN)} 
      \ \hbox{ for } v\in H^r(\RN) 
      \ \hbox{ and } r \in (0,1).
\end{equation}
In particular, the norm in $H^r(\RN)$ can be equivalently expressed as
\begin{equation}\label{eq:normaHsigbis}
   \|v\|^2_{H^r(\RN)} = \|v\|^2_{L^2(\RN)} + \| \Arm v\|^2_{L^2(\RN)} 
  = \|v\|^2_{L^2(\RN)} + \big\| \vert\xi\vert^r \hat{v} \big\|^2_{L^2(\RN)}.
\end{equation}
It is worth noting that there is another
possible approach for dealing with fractional 
Laplacians on bounded domains. 
Indeed, one may define the operator, called
{\it spectral fractional Laplacian}, as
\begin{equation} \label{eq:spectral}
  (-\Delta_{\Omega})^r f(x) := \sum_{j=1}^{+\infty}
    \lambda_{j}^r\hat f_{j}\phi_{j}(x),\ x\in \Omega,
\end{equation}
where $\lambda_{j}>0, j=1,2,\ldots$, are
the eigenvalues of the Dirichlet Laplacian
on $\Omega$, $\phi_j$ are the corresponding normalized
eigenfunctions, and
$$
  \hat f_j :=
   \int_{\Omega} f(x)\phi_j(x) \,\dx, \ \hbox{ with } \| \phi_j \|_{L^2(\Omega)} = 1.
$$
As observed in \cite{serva-valdi_eigen}, the spectral
fractional Laplacian $(-\Delta_\Omega)^r$ is a different operator
with respect to the operator considered in this paper.
We refer to \cite{cabre_tan} and \cite{Bonf_Sire_Vaz14} and
to the references therein for the functional framework 
related to \eqref{eq:spectral} and for the analysis
of some differential problems 
involving $(-\Delta_\Omega)^r$.


\subsection{An $L^2$-framework for fractional Laplacians.}
\label{Ss:L2}
The weak fractional Laplacian $\Ar$ can be also interpreted in the
framework of convex analysis. Actually, for $r\in(0,1)$, 
we can introduce the functional $G_r:\LO\to [0,+\infty]$ given by
\begin{equation}\label{defiG}
  G_r(v) := 
   \begin{cases} \displaystyle
     \frac{C(r)}{4} \iint_{Q} K_r(x-y) | v(x)-v(y) |^2 \, \dx \, \dy
      & \text{ if } v\in \Xzr,\\
      + \infty & \text{ otherwise}.
   \end{cases}
\end{equation}
Then, it is obvious that $G_r$ is a lower semicontinuous
(in $\LO$), convex functional. Moreover,
given $u\in D(G_r) := \Xzr$ (the {\it effective domain}\/ of $G_r$),
it is clear that, for all $v\in \Xzr$,
\begin{equation}\label{gateaux1}
  \lim_{t\to 0} \frac{G_r( u + tv ) - G_r(u)}{t}
    = \left\langle \Ar u, v \right\rangle_{\Xzr}.
\end{equation}
On the other hand, if $\xi\in \LO$ belongs to $\partial G_r(u)$, 
where the
 {\it subdifferential} $\partial G_r$ is defined by
$$
\partial G_r(w) := \{\xi \in \LO \colon G_r(v)-G_r(w) \geq (\xi, v-w) \
 \mbox{ for all } \ v \in \LO\}
\quad \mbox{ for } \ w \in D(G_r)
$$
with domain $D(\partial G_r) := \{w \in D(G_r) \colon \partial G_r(w)
 \neq \emptyset\}$, then one
can easily check that, again for all $v\in \Xzr$,
\begin{equation}\label{gateaux2}
  \lim_{t\to 0} \frac{G_r( u + tv ) - G_r(u)}{t}
    = ( \xi , v ).
\end{equation}
Combining \eqref{gateaux1} and \eqref{gateaux2}, for $u \in D(\partial
   G_r)$ and $\xi \in \partial G_r(u)$, we obtain
\begin{equation}\label{gateaux3}
  ( \xi , v )
   = \left\langle \Ar u, v \right\rangle_{\Xzr}
   \ \text{ for all } v\in \Xzr,
\end{equation}
whence $\partial G_r(u) = \{\xi\}$ and $\xi$ is the realization in $\LO$
of $\Ar u$. Then $\partial G_r$ is unbounded linear in $\LO$.
From now on, $\partial G_r$ will be denoted, with a small abuse of
notation, by $\Ar$. Then $D(\Ar)$ means $D(\partial G_r)$; moreover any
$u \in D(\Ar)$ can be seen as the solution
to the elliptic problem
\begin{equation}\label{gateaux4}
\Ar u = f \ \mbox{ in } \LO
\end{equation}
for some $f \in \LO$. Here we note that \eqref{gateaux4} does not mean
that $\Ar u$ vanishes outside $\Omega$ (see \S \ref{Ss:FF}).
Due to the lack of regularity results for fractional elliptic
problems such as \eqref{gateaux4}, it is a nontrivial issue to give a
precise characterization to the domain $D(\Ar)$. Up to our knowledge,
the sharpest results available in literature are due to Ros-Oton and
Serra, who proved in~\cite[Prop.~1.4 (ii)-(iii)]{ROS} the following:
\begin{propo} \label{prop:ROS}
 Let $\Omega \subset \RR^N$ be a bounded $C^{1,1}$-domain.
 Then, if $r \in(0, \frac N 4 ) \cap (0,1)$, the solution $u$ to
 \eqref{gateaux4} satisfies 
 \begin{equation}\label{ROS11}
   \| u \|_{L^q(\Omega)} \le C \| f \|_{\LO}
    \ \text{ for }\, q = \frac{2N}{N-4r}, 
 \end{equation}
 while for $r\in(\frac N 4,1) \cap (0,1)$ we have
 \begin{equation}\label{ROS12}
   \| u \|_{C^\alpha(\overline\Omega)} \le C \| f \|_{\LO}
    \ \text{ for }\, \alpha = \min \left\{ r , 2r-\frac N 2
					\right\}. 
 \end{equation}
 In both cases the constant $C>0$ depends only on 
 $r$, $|\Omega|$, and $q$ {\rm (}or $\alpha${\rm )}.
\end{propo}
%
%
The operator $\Ar$ also enjoys a useful monotonicity property. Indeed, let 
$\beta:\RR \to \RR$ be any smooth monotone
function such that $\beta(0)=0$ and let 
$\hbeta:\RR\to [0,+\infty]$ be the (convex) function
such that $\hbeta(0)=0$ and $\hbeta'=\beta$. Then,
setting for instance
\begin{equation}\label{trunc}
  \beta\ee(r):=
    \begin{cases}
      \beta(r) & \text{if } \ |r|\le \epsi^{-1},\\
      \beta(\epsi^{-1}) & \text{if } \ r > \epsi^{-1},\\
      \beta(-\epsi^{-1}) & \text{if } \ r < -\epsi^{-1},
    \end{cases}
\end{equation}
then, for fixed $\epsi>0$, $\beta\ee$ is bounded,
Lipschitz continuous, and monotone. Hence, it is 
immediate to check that, if $v\in\Xzr$, then 
$\beta\ee(v)\in \Xzr$ for each $\vep > 0$. Hence, by 
monotonicity,
\begin{equation}\label{eq:monotone:e}
  \langle \Ar v, \beta_\epsi(v) \rangle_{\Xzr} 
    = \frac{C(r)}{2} \iint_{\RdN} K_r(x-y) (v(x)-v(y))(\beta\ee(v(x))-\beta\ee(v(y)))\,\dx \,\dy
    \ge 0.
\end{equation}
Moreover, it is clear that, for any $r_1,r_2\in \RR$,
the product $(\beta\ee(r_1) - \beta\ee(r_2))(r_1-r_2)$
increases as $\epsi$ decreases. Hence, letting
$\epsi\searrow 0$, we can apply 
the monotone convergence theorem to \eqref{eq:monotone:e}.
This gives
\begin{align}\label{eq:monotone}
  & \frac{C(r)}{2} \iint_{\RdN} K_r(x-y) (v(x)-v(y))(\beta\ee(v(x))-\beta\ee(v(y)))\,\dx \,\dy \\
  \nonumber 
  & \to \frac{C(r)}{2} \iint_{\RdN} K_r(x-y) 
     (v(x)-v(y))(\beta(v(x))-\beta(v(y)))\,\dx \,\dy
    \ge 0.
\end{align}
Of course, the latter integral may well be (plus) infinity.

\bigskip

\subsection{Asymptotics of $\Dr$ and principal eigenvalues as $r \searrow 0$}%
%

Let us start with the following lemma on the behavior of $\Dr$ as $r
\searrow 0$, which will play an important role in the sequel:
\begin{lemma} \label{lem:limit_s_zero}
 For any $\phi \in \Sch$, there holds
 \begin{equation}
 \label{eq:limit_s_zero}
   \Dr \phi \xrightarrow{r\searrow 0} \phi\ \hbox{ strongly in } H^\alpha(\mathbb{R}^N)
    \ \text{ for any }\alpha\ge 0.
 \end{equation}
\end{lemma}
\begin{proof}
The Plancherel identity and the definition of $\Dr$ by Fourier
 transform imply
\begin{equation}\label{eq:planche}
  \big\| \Dr \phi - \phi \big\|^2_{H^\alpha(\RN)} 
   = \big \| (1+|\xi|^2)^{\alpha/2} (\vert \xi\vert ^{2r} - 1) \hat{\phi} \big\|^2_{L^2(\RN_\xi)}.
\end{equation}
Consequently, since $\hat{\phi}$ belongs to the Schwartz class $\mathcal S(\RN)$ of rapidly
decreasing functions, we can find a positive 
$L^1(\RN_\xi)$-function $g$ such that 
$$ 
  \big\vert (1+|\xi|^2)^\alpha (\vert\xi\vert^{2r}-1)^2\hat{\phi}^2(\xi)\big\vert\le g(\xi)\
   \ \text{ for all } \xi\in\RNxi.
$$
Then, noting that $ (1+|\xi|^2)^\alpha (\vert\xi\vert^{2r}-1)^2\hat{\phi}^2(\xi)\xrightarrow{r\searrow 0} 0$
for any $\xi\in \RNxi$, we obtain \eqref{eq:limit_s_zero} via the dominated convergence
theorem and \eqref{eq:planche}.
\end{proof}

It is also important to recall the following (asymptotic) relation 
between the $H^r$- and the $L^2$-norm 
(see \cite{Dine_Pala_Vald} and \cite{mazya-shapo}),
\begin{equation}\label{eq:HsL2}
  \lim_{r\searrow 0}\frac{C(r)}{2}[v]^2_{H^r(\RN)} 
   = \|v\|^2_{L^2(\RN)} 
     \ \hbox{ for any } \ v\in \bigcup_{\tau\in (0,1)} \mathcal{X}_{\tau,0},
\end{equation}
which is clearly related to Lemma \ref{lem:limit_s_zero}. 
%
%
%
%
%
%
%
%
%

%
%
Next, we shall characterize the behavior of the (weak) fractional 
Laplacian $\Ar$ as the index $r$ goes to 0 
(cf.~Lemma \ref{lem:limit_s_zero}). Indeed, we 
can prove that $\Ar$ tends to the identity operator in a
suitable way.
\begin{lemma} \label{cor:weak_limit_s_zero}
 Let $\{r_k\} \subset (0,1)$ be a sequence with 
 $r_k\searrow 0$ as $k\nearrow +\infty$.
 Let $\{v_k\} \subset \LO$ be a sequence such that $v_k \in D(\Ark)$ 
 for all $k\in \NN$. Moreover, let us assume both $\{v_k\}$ 
 and $\{\Ark v_k\}$ to be uniformly bounded in $\LO$,
 respectively. 
 Then, denoting by $v$ the weak limit of $v_k$ 
 in $\LO$ (up to a non-relabeled subsequence of $k$),
 we have
 \begin{equation} \label{eq:weak_limit_s_zero}
   \Ark v_k \xrightarrow{k\nearrow +\infty} v \ \hbox{ weakly in } 
    \LO. 
 \end{equation}
\end{lemma}
\begin{proof}
The boundedness of $\Ark v_k$ in $\LO$  
entails the existence of $w \in \LO$ such 
that 
\begin{equation*}
  \Ark v_k \xrightarrow{k\nearrow +\infty} w \ \hbox{ weakly in } 
   \LO, 
\end{equation*}
up to a non-relabeled subsequence of $k$ (of course, we can assume
that this subsequence is extracted from the subsequence along 
which $v_k$ weakly converges to $v$ in $\LO$). 
To conclude, we have to prove that $w \equiv v$ a.e.~in $\Omega$.
We test the weak convergence $\Ark v_k\to w$ 
by $\vf\in C^{\infty}(\RN)$ with support in $\Omega$.
Recalling Lemma~\ref{lem:limit_s_zero} and \eqref{eq:distr_lapl3},
we then have
\begin{align*}
  \int_{\RN} w \vf \, \dx 
  & = \lim_{k\nearrow +\infty} \int_{\RN} \Ark v_k \vf \,\dx \\
 %
  & = \lim_{k\nearrow +\infty} \int_{\RN}v_k \Ark \vf\, \dx
  = \int_{\RN} v \vf \, \dx , 
\end{align*} 
which implies $w\equiv v$ almost everywhere in $\RN$.
\end{proof}
\noindent%

The optimal constant of the Poincar\'e-type inequality
\eqref{eq:poincare} 
depends on the first eigenvalue $\lambda_1(r)$
of the fractional eigenvalue problem 
\begin{equation}\label{eq:eigen}
  v \in \Xzr, \quad
   \Ar v = \lambda v \ \hbox{ in } \LO, 
\end{equation}
which also does not mean that $\Ar v$ vanishes outside $\Omega$ (see \S
\ref{Ss:FF}). More precisely, the optimal value of $c_P(r)$ is given as
$1/\lambda_1(r)$, that is,
\begin{equation}\label{eq:poincare-opt}
\|v\|_{\LO}^2 \leq \dfrac 1 {\lambda_1(r)} \|v\|_{\Xzr}^2
\quad \mbox{ for all } \ v \in \Xzr.
\end{equation}
The first eigenvalue $\lambda_1(r)$ 
is characterized in the next lemma, which comprises results
from~\cite{serva-valdi11} and~\cite{serva-valdi_eigen}
and clarifies the spectral properties of the fractional Laplacian.
%
\begin{lemma}[\cite{serva-valdi11},~\cite{serva-valdi_eigen}]\label{lem:1eigen}
 Let $r \in (0,1)$. Then the fractional eigenvalue problem 
 \eqref{eq:eigen} admits a first eigenvalue $\lambda_1(r)$ which is 
 strictly positive, simple, isolated, and can be characterized as 
 \begin{equation} \label{eq:cara_eigen}
   \lambda_{1}(r)
     = \min_{u\in \Xzr, \|u\|_{\LO}=1} \frac{C(r)}{2}[u]^2_{H^r(\RN)}
     = \min_{u\in \Xzr \setminus\left\{0\right\}}  
          \frac{C(r)}{2} \frac{[u]^2_{H^r(\RN)}}{\|u\|^2_{\LO}}.
 \end{equation}
 Moreover, there exists a unique positive first
 eigenfunction $e_1\in \Xzr$
 which satisfies $\| e_1 \|_{\LO} = 1$, and 
 attains the minimum in \eqref{eq:cara_eigen}.

 Finally, denoting with $\lambda_1$ the first eigenvalue of 
 the (standard) Dirichlet problem  
 \begin{equation}\label{eq:eigen_laplacian}
   - \Delta v = \lambda v \ \hbox{ in } \Omega,\quad
      v = 0 \ \hbox{ on } \partial\Omega,
 \end{equation}
 then for any $r \in (0,1)$ there holds 
 \begin{equation} \label{eq:eigen_upper}
   \lambda_1(r) < \lambda_1^r.
 \end{equation}
\end{lemma}

For our purposes it will also be extremely important
to understand the asymptotics of $\lambda_1(r)$ with respect 
to $r\searrow 0$. The following proposition clarifies the situation. 
In particular, we prove that, as $r\searrow 0$, the corresponding sequence of 
first eigenvalues of $\Ar$ tends to $1$. Note that, being $\Ar$
an approximation of the identity as $r\searrow 0$, then the result
is not completely unexpected. However, at least to our knowledge,
this was not observed before. Moreover, due to some lack in compactness
in the sequence of the first eigenfunctions of $\Ar$ the proof 
is not straightforward and requires some precise estimates.

\smallskip
\begin{propo}
\label{prop:asy_1eigen}
 Let $\{r_k\}\subset (0,1)$ satisfy $r_k\searrow 0$ as $k\nearrow +\infty$.
 Then we have
 \begin{equation}\label{eq:lim1eigen}
   \lim_{k\nearrow +\infty}\lambda_1(r_k) =  1.
 \end{equation}
\end{propo}
\begin{proof}
To simplify the notation, for any $k\in \mathbb{N}$
we denote by $\lambda_k$ and $v_k$ 
the first eigenvalue and the (normalized in $\LO$) first
 eigenfunction of \eqref{eq:eigen}, respectively. Then,
the couple $(\lambda_k, v_k)$ 
solves the eigenvalue problem
\begin{equation} \label{eq:eigen1}
  \Ark v_k = \lambda_k v_k \ \hbox{ in } \LO.
\end{equation}
As a consequence, the only uniform (in $k$) estimate
on which we can rely is $\|v_k\|_{L^2(\RN)}=1$. 
Hence, the only convergence we expect
for $v_k$ is a \emph{weak} $L^2$-convergence. Indeed, this is enough
to pass to the limit in \eqref{eq:eigen1} 
(recall Lemma~\ref{cor:weak_limit_s_zero}).
On the other hand, since a priori we cannot exclude that the weak limit 
of $v_k$ is zero, we are not able, at this level, to conclude that 
$\lambda_k\to 1$. For this reason, we have to use a different approach.
First of all, thanks to \eqref{eq:eigen_upper} 
of Lemma~\ref{lem:1eigen}, we only need to prove a 
lower bound for the limit in \eqref{eq:lim1eigen}, namely
\begin{equation}\label{eq:lowelim1eigen}
  \liminf_{k\nearrow +\infty}\lambda_k\ge 1.
\end{equation}
To this aim, we test \eqref{eq:eigen1} by $v_k$. Using that $\|v_k\|_{\LO}=1$
and recalling \eqref{eq:distr_lapl3} and \eqref{eq:semi_norm}, we get
\begin{align}\label{eq:eigenproof1}
  \lambda_k & = \lambda_k \|v_k\|^2_{\LO}
 = \langle \Ark v_k, v_k \rangle_{\Xzrk}\\
 \nonumber 
  & = \int_{\RN}\vert \Drkm v_k \vert^2 \, \dx
    = \int_{\RN_\xi}\vert \xi\vert^{2 r_k}\vert\hat{v}_k(\xi)\vert^{2} \, \hbox{d}\xi.
\end{align}
For each $k\in \mathbb{N}$, set $f_k(\xi)=\vert
 \hat{v}_k(\xi)\vert^{2}$.
Then $f_k\ge 0$ a.e.~in $\RN_\xi$ and, by Plancherel's identity,
$f_k\in L^1(\RN_\xi)$ with $\int_{\RN_\xi}f_k(\xi)\,\hbox{d}\xi=1$
for any $k\in\mathbb{N}$. Moreover, 
by~\eqref{eq:eigenproof1} and \eqref{eq:eigen_upper},
the sequence of the $2 r_k$-moments of $f_k$ is uniformly
bounded with respect to $k$. Namely, we have
\begin{equation} \label{eq:mom_finite}
  \int_{\RN_\xi}\vert \xi\vert^{2 r_k}f_k(\xi) \, \hbox{d}\xi
   = \lambda_k<(\lambda_1)^{r_k}.
\end{equation}
Now, combining this information with the fact that 
$v_k$ is zero outside $\Omega$, we can easily deduce that $f_k$ lies
in $L^\infty(\RN_\xi)$. Indeed, using $\vert e^{-ix\cdot \xi}\vert \le 1$ 
together with the H\"older inequality, we obtain
\begin{align} \label{eq:Linftyf}
  f_k(\xi) & = \frac{1}{(2\pi)^N} \left| \int_{\RN} e^{-ix\cdot \xi} v_k(x)\,\dx \right|^2 
   = \frac{1}{(2\pi)^N}\left| \int_{\Omega} e^{-ix\cdot \xi} v_k(x) \,\dx\right|^2 \\
 \nonumber
  & \le \frac{|\Omega|}{(2\pi)^N} \|v_k\|^2_{\LO} 
   = \frac{\vert \Omega\vert}{(2\pi)^N}, \quad \mbox{ for all } \, \xi \in \RN_\xi.
\end{align}
At this point, we need to estimate the $L^1$-norm of $f_k$
in a quantitative way. To this purpose, we use the following
interpolation lemma, whose proof is based on a technique widely
used in the kinetic theory community (alternatively, one could 
refer to \cite[Lemma 2.1]{yolcu_yolcu},
where a similar inequality is shown in a different way):
\begin{lemma}\label{lemma:interp}
 Given $f\in L^1(\RN)\cap L^\infty(\RN)$ with
 $\int_{\RN}\vert x\vert^\alpha \vert f(x)\vert \,\dx < +\infty$,  
 where $\alpha>0$, there holds
 \begin{equation}\label{eq:inter}
   \| f \|_{L^1(\RN)}\le \kappa (N,\alpha) 
    \big\| \vert x\vert^\alpha f \big\|_{L^1(\RN)}^{\frac{N}{N+\alpha}}
    \|f \|_{L^\infty(\RN)}^{1-\frac{N}{N+\alpha}},
 \end{equation} 
 where 
 $$
   \kappa(N,\alpha) = \alpha^{-\frac{\alpha}{\alpha +N}}(\alpha+
 N)N^{-\frac{N}{N+\alpha}} d^{\frac{\alpha}{N+\alpha}}
 $$
 and $d = d(N)$ is the volume of the unit ball in $\RN$.
\end{lemma}
\begin{proof}
To simplify the presentation, we set $L:=\|f \|_{L^\infty(\RN)}$ and 
$M:=\int_{\RN}\vert x\vert^\alpha \vert f(x)\vert \,\dx$.
We compute, for $R>0$ to be chosen later, 
\begin{equation}\label{eq:inter1}
  \int_{\RN}\vert f\vert \,\dx 
   = \int_{\vert x\vert\le R} \vert f\vert \, \dx 
      + \int_{\vert x\vert> R}\vert f\vert \,\dx 
   \le d R^N L + R^{-\alpha}M.
\end{equation}
Now, optimizing with respect to $R$ the function 
$g(R):= d R^N L + R^{-\alpha}M$, we find 
\begin{equation}\label{eq:inter2}
  \| f \|_{L^1(\RN)} \le \inf_{R>0}g(R) 
   = g(\bar R), \ \hbox{ with } \bar R 
   = \Big( \frac{\alpha M}{dNL} \Big)^{1/(N+\alpha)}.
\end{equation}
Hence, computing $g(\bar R)$, we readily get \eqref{eq:inter}.
\end{proof}
\noindent
{\it Continuation of Proof of Proposition \ref{prop:asy_1eigen}.} %
Applying the lemma, with $f=f_k$, $\alpha=2 r_k$, 
$\| \vert \xi\vert^{2r_k} f_k\|_{L^1(\RN_\xi)} = \lambda_k$
and $\|f_k\|_{L^\infty(\RN)}\le \frac{\vert\Omega\vert}{(2\pi)^N}$,
we then get
\begin{equation}\label{eq:L1}
  \|f_k\|_{L^1(\RN_\xi)}
   \le \kappa(N,2r_k) \lambda_k^{\frac{N}{N + 2 r_k}} 
      \left(
       \frac{\vert\Omega\vert}{(2\pi)^N}
       \right)^{\frac{2r_k}{N+2r_k}},
\end{equation}
which, together with the fact that $\|f_k\|_{L^1(\RN_\xi)}=1$ 
by construction, implies
\begin{equation}\label{eq:lower_lambdak}
  \lambda_k \geq 
\kappa(N,2r_k)^{-\frac{N+2r_k}N} \left(
\dfrac{(2\pi)^N}{|\Omega|}
\right)^{\frac{2r_k}N}.
\end{equation}
Hence, letting $k\nearrow +\infty$ and noting that 
$\lim_{k\nearrow +\infty}\kappa(N,2r_k)=1$, we get
 \eqref{eq:lowelim1eigen}, which implies the thesis. 
\end{proof}


\section{Assumptions and statement of main results}
\label{sec:main}

Let us start with listing our hypotheses 
on the data of the problem. First of all, we make precise
the choice of the confining potential by choosing
\begin{equation}\label{eq:confining}
  W(v) = \frac1p |v|^p - \frac12 v^2,
   \quad p\in (1,\infty) \setminus \{2\}.
\end{equation}
We also set $\hat\beta(v):= \frac1p |v|^p$ and 
$\beta(v):=\hat\beta'(v) = |v|^{p-1}\sign v$, for brevity.
In general, we note the case $p>2$ as the {\it coercive}\/
case, while for $p\in(1,2)$ we speak of a {\it non-coercive}\/
potential. Indeed, $W$ is unbounded from below in the latter
situation.
As noted in Introduction, the coercive case includes,
up to an additive constant, the standard double well potential
\begin{equation}\label{eq:double_well} 
  W(v)=\frac{1}{4}(v^2-1)^2.
\end{equation}
Next, for $v\in \LO$, we introduce the energy functional
\begin{equation}\label{eq:def_energy}
  \mathbb{E}_{\sigma}(v):= 
   \frac12 \| v \|_{\Xzsi}^2
    + \int_{\Omega} W(v) \, \dx,
\end{equation}
whenever it makes sense. In particular,
in the case $p>2$, the {\it domain}\/ of $W$,
i.e., the set where it takes finite values, is given by
\begin{equation}\label{defiEsigma}
  \Esig = \big\{v \in \LO \hbox{ such that } \Esi(v) < +\infty\big\}
   = \Xzsi \cap L^p(\RR^N).
\end{equation}
Of course, $\Esig$ is a Banach space with the natural norm
\begin{equation}\label{norEsigma}
  \| v \|_{\Esig} := \| v \|_{\Xzsi} 
   + \| v \|_{L^p(\RN)}.
\end{equation}
If $p\le \frac{2N}{N-2\sigma}$, then by Sobolev's embeddings
$\Esig$ coincides in fact with $\Xzsi$ (hence we can use
the norm of $\Xzsi$ instead of \eqref{norEsigma}).
On the other hand, in the non-coercive case $p\in(1,2)$, the
functional $\Esig$ may be unbounded from below. 
In that situation, existence and uniqueness of 
(global in time) solutions still hold.
However, when dealing with the singular limits, 
some difficulties arise from the lack of coercivity.

\bigskip

In what follows, we will look for solutions taking values 
in the energy space $\Esig$. Correspondingly, we ask that the
same regularity is satisfied by the initial datum:
\begin{equation}\label{eq:cond_iniz}
  u_0 \in \Esig.
\end{equation}
Next, we introduce a weak (energy) formulation of the system
\eqref{eq:fCH}-\eqref{eq:chem_pot}. Here and henceforth, the notation
$C_w([0,T];X)$ will represent the class of continuous functions on
$[0,T]$ in the weak topology of a normed space $X$. 
\begin{defi}\label{def:weaksol}
 We say that $(u,w)$ is a weak solution to the Cauchy-Dirichlet problem\/
 \eqref{eq:fCH}-\eqref{eq:bc} for the fractional Cahn-Hilliard system
 if, for all $T>0$, we have
 \begin{align}\label{regou}
   & u \in C_w([0,T];\Esig) \cap C([0,T];\LO) \cap W^{1,2}(0,T;\Xzs'), \\
  \label{regow}
   & w \in L^2(0,T;\Xzs);
 \end{align}
 moreover, the couple $(u,w)$ satisfies, 
 a.e.~in $(0,T)$, the following weak formulation of
 \eqref{eq:fCH}-\eqref{eq:chem_pot}\/{\rm :}
 \begin{alignat}{4} \label{eq:ch_weak}
   &\partial_t u + \As w = 0 \ &\hbox{ in }& \Xzs', \\
  \label{eq:cp_weak}
   &w = \Asig u + \B(u) - u\ &\hbox{ in }& \Esig',
 \end{alignat}
where $\B(u)$ denotes the bounded linear functional on $L^p(\mathbb
 R^N)$ defined by
$$
\langle \B(u), v \rangle_{L^{p}(\mathbb R^N)} 
= \int_\Omega \beta(u(x)) v(x) \; \d x
\quad \mbox{ for } \ v \in L^p(\mathbb R^N), 
$$
and, finally, the initial condition \eqref{eq:iniz_bc} holds in the
 following sense\/{\rm :}
 \begin{equation}\label{init}
  u(t) \to u_0 \quad \mbox{ strongly in } \LO \mbox{ and weakly in }
  \Esig \mbox{ as } \ t \searrow 0.
 \end{equation}
\end{defi}

Correspondingly, we can prove 
the following existence and uniqueness result:
\begin{theorem}\label{th:eu}
 Let $s,\sigma \in (0,1)$, $p\in (1,\infty) \setminus \{2\}$, and assume\/
 \eqref{eq:confining} and\/ \eqref{eq:cond_iniz}. Then, the fractional
 Cahn-Hilliard system \eqref{eq:fCH}-\eqref{eq:bc} admits a unique
 weak solution $(u,w)$ in the sense of Definition~\ref{def:weaksol},
 which additionally satisfies
 \begin{align}\label{enineq-0}
    \|\beta(u(\cdot,t))\|_{\LO}^2
     \leq 2 \left( \|w(t)\|_{\LO}^2
         + \|u(t)\|_{\LO}^2 \right)
    \quad \mbox{for a.e.~} t \in (0,T),\\
    \beta(u) \in L^2(0,T;\LO).
  \label{regobeta}
 \end{align}
 Moreover, $u(t) := u(\cdot,t)$ and $\Esi(u(t))$ are right-continuous on
 $[0,T)$ with respect to the strong topologies of $\Esig$ and $\mathbb R$,
 respectively, and the following energy inequality holds
 true\/{\rm :}
 \begin{equation}\label{energy_inequality}
   \|w(t)\|_{\Xsz}^2 + \dfrac{\d}{\d t} \Esi(u(t)) \leq 0
    \quad \mbox{for a.e.~} t \in (0,T).
\end{equation}
In particular, if $\sigma \geq s$, then $u \in C([0,T];\Esig)$ and
 $\Esi(u(t))$ is absolutely continuous on $[0,T]$\/{\rm ;} moreover, 
 the inequality \eqref{energy_inequality} can be replaced by an
 equality, namely we have
$$
   \|w(t)\|_{\Xsz}^2 + \dfrac{\d}{\d t} \Esi(u(t)) = 0
    \quad \mbox{for a.e.~} t \in (0,T).
$$
\end{theorem}
A proof of the above result will be carried out in 
Section~\ref{sec:conti} by means of time-discretization,
a-priori estimates, and compactness arguments.
It is worth noting that, in
addition to \eqref{regou}-\eqref{regobeta}, one 
could show that weak solutions satisfy parabolic time-smoothing
properties. We will analyze this issue in a forthcoming
paper, where we also plan to consider a more general class of 
potentials~$W$.

Now, we come to the behavior of (families of) solutions
as $\sigma$ tends to $0$. Then, in the coercive case we can
prove the following 
\begin{theorem}[from Cahn-Hilliard to porous medium]\label{th:pm}
 Let $p\in(2,\infty)$, $s \in (0,1)$ and let
 $\{\sigma_k\}\subset (0,1)$ be such that 
 $\sigma_k\searrow 0$ as $k\nearrow +\infty$. Moreover, let us given
 a sequence of initial data $\{u_{0,k}\}$ and $u_0 \in \Xsz'$ satisfying
 \begin{equation}\label{hp:uzk}
   \sup_{k\in \Nz} \mathbb{E}_{\sigma_k}(u_{0,k}) < +\infty, \quad
     u_{0,k} \to u_0 \ \text{ in } \Xzs'.
 \end{equation}
 Let $(u_k,w_k)$ denote the corresponding sequence of 
 unique weak solutions to\/ 
 \eqref{eq:fCH}-\eqref{eq:bc}, with $\sigma = \sigma_k$ and initial
 datum $u_{0,k}$. Then, there exist a (non-relabeled) subsequence
 of $\{k\}$ and a pair of limit functions $(u,w)$ such that
 \begin{alignat}{4}
   u_k &\to u \quad &&\mbox{ weakly star in }
  L^\infty(0,T;\LpO),\label{pm:c:1}\\
  & &&\mbox{ strongly in } L^p(0,T;\LpO) \cap C([0,T];\Xzs'), \\
  & &&\mbox{ weakly in }
  W^{1,2}(0,T; \Xzs'),\label{pm:c:2}\\
  w_k &\to w \quad
  &&\mbox{ weakly in } L^2(0,T;\Xzs).\label{pm:c:4}
 \end{alignat}
 Moreover,
 \begin{equation}\label{pm:regu}
 u \in C_w([0,T];\LpO) \cap W^{1,2}(0,T;\Xzs'),
 \quad \beta(u) = w \in L^2(0,T;\Xzs).
\end{equation}
 Furthermore, $u$ is a (weak) solution to the fractional porous medium 
 equation
 \begin{equation}\label{eq:weak_pm}
    \partial_t u  + \As \beta(u) =  0 \ \hbox{ in } \Xzs',
     \ \text{ a.e.~in }(0,T),
 \end{equation}
 with the initial condition $u|_{t=0}=u_0$, i.e.,
\begin{equation}\label{ic-pm}
u(t) \to u_0 \ \mbox{ strongly in } \Xzs' \ \mbox{ as } \ t
 \searrow 0.
\end{equation}
\end{theorem}
Now we address  the non-coercive case $p\in (1,2)$. 
In this situation (which corresponds to the fast diffusion
range when dealing with equations of the type
\eqref{eq:weak_pm}) we can prove similar
results but the analysis requires 
some extra assumption together with a small modification
of the system. First of all, we need a compatibility 
condition between $p$ and the fractional order $s$ in \eqref{eq:ch_weak}. 
To be precise, we ask that
\begin{equation}\label{ass:comp_p_s}
  p > \frac{2N}{N + 2s} =: 2_*. 
\end{equation}
Note that the above condition implies that 
$H^s(\Omega) \hookrightarrow L^{p'}(\Omega)$ with
compact embedding. Hence $\Xzs$ is compactly embedded in
$L^{p'}(\RN)$.
Moreover, we need to modify a bit the energy functional.
Namely, we
replace $\mathbb{E}_\sigma$
with the following
\begin{equation}\label{eq:new_ene}
  \tilde{\mathbb{E}}_{\sigma}(v):= \frac12 \| v \|_{\Xzsi}^2 +
   \frac{1}{p}\|v\|^{p}_{L^p(\Omega)} 
  -\frac{\lambda_1(\sigma)}{2}\|v\|^2_{L^2(\Omega)}
  \quad \mbox{ for } \ v \in \Esig = \Xzsi,
\end{equation}
where $\lambda_1(\sigma)$ is the first eigenvalue of 
\eqref{eq:eigen} with $r = \sigma$. 
Note that, in view of $\lambda_1(\sigma)\xrightarrow{\sigma\searrow 0} 1$
(cf.~Prop.~\ref{prop:asy_1eigen}), one expects that the contribution 
of $\lambda_1(\sigma)$ in the limit is idle.
The reason why we need to replace $\mathbb{E}_\sigma$ with
$\tilde{\mathbb{E}}_{\sigma}$ lies in the fact
that, thanks to Poincar\'e's inequality~\eqref{eq:poincare-opt},
one has
\begin{equation}\label{unif:coerc}
  \tilde{\mathbb{E}}_{\sigma}(v)\ge \frac{1}{p}\|v\|^{p}_{L^p(\mathbb{R}^N)}
   \quad \text{for all } v\in \Xzsi,
\end{equation}
and independently of the value of $\sigma$. 
Of course, modifying the energy through the choice~\eqref{eq:new_ene}
leads correspondingly to a modification of the ``original'' system
\eqref{eq:ch_weak}-\eqref{eq:cp_weak}, which is now
replaced by
\begin{alignat}{4} \label{eq:ch_weak_mod}
 & \partial_t u + \As w = 0 \ &\hbox{ in }& \Xzs', \\
 \label{eq:cp_weak_mod}
 & w = \Asig u + \B(u) - \lambda_1(\sigma) u \ &\hbox{ in
 }& 
 \Xzsi', 
\end{alignat}
still complemented with the initial condition~\eqref{init}.
At fixed $\sigma$, existence and uniqueness of energy
solutions for the modified system \eqref{eq:ch_weak_mod}-\eqref{eq:cp_weak_mod} 
follow from the very same argument given for Theorem~\ref{th:eu}. 
Moreover, the uniform coercivity provided by \eqref{unif:coerc}
permits us to prove a counterpart of 
Theorem~\ref{th:pm} for the fast-diffusion case in the following
theorem:
\begin{theorem}[from Cahn-Hilliard to fast-diffusion]\label{th:fd}
 Let $s \in (0,1)$, $p\in(2_*,2)$, $2_*$ being given
 by~\eqref{ass:comp_p_s}, and let $\{\sigma_k\}$
 be as before. Moreover, let us given
 a sequence of initial data $\{u_{0,k}\}$ and $u_0 \in \Xzs'$
 satisfying 
 \begin{equation}\label{hp:uzk_mod}
   \sup_{k\in \Nz} \tilde{\mathbb{E}}_{\sigma_k}(u_{0,k}) < +\infty, \quad
   u_{0,k} \to u_0 \ \text{ in } \Xzs'. 
 \end{equation} 
 Let $(u_k,w_k)$ denote the corresponding sequence of 
 unique weak solutions of \eqref{eq:fCH}-\eqref{eq:bc},
 with $\sigma = \sigma_k$, $W'(v) = \beta(v) - \lambda_1(\sigma_k) v$
 and initial datum $u_{0,k}$. Then, 
 there exist a (non-relabeled) subsequence
 of $\{k\}$ and a pair of limit functions $(u,w)$ satisfying
 \eqref{pm:c:1}-\eqref{pm:regu}. 
 Moreover, $u$ is a (weak) solution to the fractional fast-diffusion
 equation
 \begin{equation}\label{eq:weak_fd}
    \partial_t u  + \As \beta(u) =  0 \ \hbox{ in } \Xzs',
     \ \text{ a.e.~in }(0,T),
 \end{equation}
 with the initial condition \eqref{ic-pm}.
\end{theorem}
%
%
Finally, we investigate the behavior of weak solutions to
\eqref{eq:fCH}-\eqref{eq:bc} 
when the order $s$ of the fractional Laplacian in 
\eqref{eq:ch_weak} is let tend to $0$. 
In this case, we can prove the following
\begin{theorem}[from Cahn-Hilliard to Allen-Cahn]\label{th:ac}
 Let $p\in(1,\infty)\setminus\{2\}$, $\sigma \in (0,1)$ and let
 $\{s_k\}\subset (0,1)$ be such that 
 $s_k\searrow 0$ as $k\nearrow +\infty$.
 Let also $u_{0,k} \in \Xzsi$ and $u_0 \in \LO$ satisfy
\begin{equation}\label{hypo-ac}
\sup_k \left( \Esi(u_{0,k}) + \|u_{0,k}\|_{\Xzsk'}^2 \right) < \infty,
 \quad u_{0,k} \to u_0 \quad \mbox{ strongly in } \LO.
\end{equation}
Let us denote as $(u_k,w_k)$ the 
 corresponding sequence of unique weak solutions 
 to \eqref{eq:fCH}-\eqref{eq:bc} with $s = s_k$ and initial datum
 $u_{0,k}$. 
 Then, there exist a (non-relabeled) subsequence
 of $\{k\}$ and a pair of limit functions $(u,w)$ such that
 \begin{alignat*}{4}
       u_k &\to u \quad 
       &&\mbox{ weakly star in } L^\infty(0,T;\Esig),\\
       & &&\mbox{ strongly in } C([0,T]; \LO),\\
       & &&\mbox{ weakly in } W^{1,2}(0,T; \Esig'),\\
       w_k &\to w \quad
       &&\mbox{ weakly in } L^2(0,T;\LO),\\
       \mathfrak{A}_{s_k} w_k &\to w \quad
       &&\mbox{ weakly in } L^2(0,T;\Esig').
      \end{alignat*}
 %
 Moreover,
 $$
 u \in C_w([0,T];\Esig) \cap W^{1,2}(0,T;\Esig'),
 \quad w \in L^2(0,T;\LO),
 $$
  $w = \Asig u + \B(u) - u$, 
 and $u$ is a (weak) solution to the fractional 
 Allen-Cahn equation 
 \begin{equation}\label{eq:weak_AC}
    \partial_t u  + \Asig u + \B(u) - u =  0 \ \hbox{ in } 
     \Esig',\ \text{ a.e.~in }(0,T),
 \end{equation}
 with the initial condition $u|_{t=0}=u_0$, i.e.,
$$
u(t) \to u_0 \quad \mbox{ strongly in } \LO \ \mbox{ and weakly in } \Esig  \
 \mbox{ as } \ t \searrow 0.
$$ 
\end{theorem}
\noindent
Theorems~\ref{th:pm}, \ref{th:fd} and \ref{th:ac}
will be proved in Section~\ref{sec:conve} below.
Actually, while the proof of Theorem~\ref{th:ac} is almost
straightforward, taking the limit $\sigma\searrow 0$ 
will require a careful analysis of the behavior of the 
eigenvalues of~$\Asigk$ as $\sigma_k\searrow 0$.
The proof of Theorem \ref{th:fd} is simplified 
by the fact that we will not start from the original system
\eqref{eq:ch_weak}-\eqref{eq:cp_weak},
but rather from its modified version 
\eqref{eq:ch_weak_mod}-\eqref{eq:cp_weak_mod},
whose energy is bounded from below, uniformly w.r.t.~$\sigma$,
thanks to \eqref{unif:coerc}.
As will be clarified from the proof,
%
this fact makes the analysis simpler. Indeed, 
we will not need to use the results 
in Prop.~\ref{prop:asy_1eigen} regarding the 
asymptotics of the first eigenvalue of the fractional
Laplacian, which, instead, are essential for proving 
Theorem~\ref{th:pm}. Actually, we expect that the 
convergence to the fast diffusion equation 
could hold, still for $p\in (2_*,2)$, also for the non-modified
fractional Cahn-Hilliard \eqref{eq:ch_weak}-\eqref{eq:cp_weak}. However,
for the moment, this remains as an open question.

%
%

\section{Existence and uniqueness}
\label{sec:conti}

This section is devoted to giving a proof of Theorem~\ref{th:eu}, i.e.,
existence and uniqueness of weak solutions along with some regularity
properties and energy inequalities. In the existence part,
the basic strategy of a proof is more or less standard, and it
consists of time-discretization, a priori estimates, compactness
arguments to obtain convergence, and Minty's trick 
for the identification of the limit. On the other hand,
we will face some difficulty in deriving the energy inequality
\eqref{energy_inequality}, particularly, the differentiability of the
energy $t \mapsto \Esi(u(t))$, and for proving the right-continuity of
$u(t)$ in the strong topology of $\Esig$. These difficulties 
are due to the simultaneous presence of two fractional Laplacians
of (possibly) different order. However, the energy inequality 
seems to be one of the minimum requirements for the
analysis of singular limits in latter sections and 
also for the investigation of the long-time behavior of global
(in time) solutions.  Moreover, the right-continuity of solutions
plays a crucial role to prove the convergence of
each global solution as $t$ goes
to infinity. More details on the long-time behavior of 
solutions will be discussed in a forthcoming paper. 
Here, it is also worth emphasizing that, in contrast 
with usual studies on the Cahn-Hilliard equation, our
argument does not rely on the coercivity of the 
double-well potential $W$. In other words, we can treat the
cases $p > 2$ and $p \in (1,2)$, simultaneously. This fact 
potentially means that our proof could also be extended to
fractional Cahn-Hilliard equations with more 
general classes of potentials.

\medskip
%
\subsection{Uniqueness of solution}%
We first note that
$\As : \Xsz \to \Xsz'$
is invertible, i.e., there exists the inverse mapping 
$\As^{-1} : \Xsz' \to \Xsz$ of $\As$. Indeed, the functional
on $\Xsz$ defined by
\begin{equation}\label{dual_Xzs}
  v \mapsto \dfrac 1 2 \|v\|_{\Xsz}^2
\end{equation}
is smooth, convex and coercive in $\Xsz$. Hence its
Fr\'echet derivative $\As$ is invertible. 
Furthermore, one observes by the definition of norms in dual spaces that
$$
 \langle \As u, u \rangle_{\Xzs} = \|u\|_{\Xzs}^2 = \|\As u\|_{\Xzs'}^2
  \quad \mbox{ for all } \ u \in \Xzs.
$$
Hence $\As$ is a duality mapping between $\Xzs$ and $\Xzs'$.

Now, let $(u_1,w_1)$ and $(u_2,w_2)$ be weak solutions of the fractional
 Cahn-Hilliard equation for the same initial data $u_0$.
Let $\hat u := u_1 - u_2$ and $\hat w := w_1 - w_2$. Then we have
\begin{equation}\label{uniq-eqn}
\partial_t \hat u + \As \hat w = 0 \ \mbox{ in } \Xsz', \quad
 \hat w = \Asig \hat u + \B(u_1) - \B(u_2) - \hat u \
\mbox{ in } \Esig'.
\end{equation}
Applying $\As^{-1} : \Xzs' \to \Xzs$ to both sides of the first
 equation of \eqref{uniq-eqn}, we have
$$
\dfrac{\d}{\d t} \As^{-1} \hat u + \hat w = 0 \ \mbox{ in
 } \Xsz.
$$
Test it by $\hat u \in \Xsz \hookrightarrow \Xsz'$ (see \eqref{H-tr}) to get
\begin{equation}\label{xJsu}
\dfrac 1 2 \dfrac{\d}{\d t} \|\hat u\|_{\Xsz'}^2
+ \langle \hat u, \hat w \rangle_{\Xsz} = 0.
\end{equation}
Here we used the fact that
\begin{align}\label{dual-test}
\left\langle \hat u, \dfrac{\d}{\d t} \As^{-1} \hat u
 \right\rangle_{\Xsz}
&=
\left\langle \As \circ \As^{-1} \hat u, \dfrac{\d}{\d t} \As^{-1}
 \hat u
 \right\rangle_{\Xsz}\\
&= \dfrac 1 2 \dfrac{\d}{\d t}
\left\|
\As^{-1} \hat u
\right\|_{\Xsz}^2
= \dfrac 1 2 \dfrac{\d}{\d t}
\left\|
\hat u
\right\|_{\Xsz'}^2,\nonumber
\end{align}
since $\As^{-1}$ is a duality mapping between $\Xsz'$ and $\Xsz$.
Test now the second equation of \eqref{uniq-eqn} by $\hat u$. Then, by 
 the monotonicity of $\beta$, it follows that
\begin{align}\label{xw}
\langle \hat w, \hat u \rangle_{\Esig}
&= \|\hat u\|_{\Xsigz}^2 + \int_\Omega \left( \beta(u_1) - \beta(u_2)
 \right) \hat u ~ \d x - \|\hat u\|_{L^2(\Omega)}^2
\\
&\geq
\|\hat u\|_{\Xsigz}^2 
 - \|\hat u\|_{L^2(\Omega)}^2.\nonumber
\end{align}
We also note that, for any functions $u \in \Esig
$ ($\hookrightarrow \Xzs'$ by \eqref{H-tr}) and $w \in \Xzs$
($\hookrightarrow \Esig'$),
\begin{equation}\label{convert_pair}
 \langle u, w \rangle_{\Xzs} 
= \int_{\RN} u(x) w(x) \; \d x
= \langle w, u \rangle_{\Esig}.
\end{equation}

Combining \eqref{xJsu} and \eqref{xw} with
\eqref{convert_pair}, we deduce that
$$
\dfrac 1 2 \dfrac{\d}{\d t} \|\hat u\|_{\Xsz'}^2
+ \|\hat u\|_{\Xsigz}^2 
\leq \|\hat u\|_{L^2(\Omega)}^2
\leq \dfrac 1 2 \|\hat u\|_{\Xsigz}^2 + C \|\hat u\|_{\Xsz'}^2
$$
for some constant $C \geq 0$. Here we used Ehrling's lemma
(see~\cite[Lemma 8]{Simon}), 
i.e., for each $\vep > 0$, one can take a constant $C_\vep \geq 0$
such that
\begin{equation}\label{ehrling-1}
\|v\|_{L^2(\Omega)}
= \|v\|_{\LO}
\leq \vep \|v\|_{\Xsigz} + C_\vep \|v\|_{\Xsz'}
\quad \mbox{ for all } \ v \in \Xsigz.
\end{equation}
Thus we get
$$
\dfrac 1 2 \dfrac{\d}{\d t} \|\hat u\|_{\Xsz'}^2
+ \dfrac 1 2 \|\hat u\|_{\Xsigz}^2 
\leq C \|\hat u\|_{\Xsz'}^2,
$$
which along with the fact that $\hat u(0) = 0$ in $\Xzs'$
implies $\hat u(t) = 0$ in $\Xzs'$ for all $t \geq
0$. Consequently, since $u_1(t)$ and $u_2(t)$ belong to $\LO$, one can
assure that $u_1(t) = u_2(t)$
a.e.~in $\RN$ for all $t \in [0,T]$, that is, $u_1 = u_2$ as desired.

\medskip
%
\subsection{Time-discretization}%
Let $N \in \mathbb N$ and let $\tau = T/N$ be a time step. 
In order to construct weak solutions, we first carry out the
following discretization of \eqref{eq:ch_weak}-\eqref{eq:cp_weak}:

\begin{alignat}{4}
& \dfrac{u_n - u_{n-1}}{\tau} + \As w_n = 0 \ &\mbox{ in }&
 \Xsz',\label{td1}\\
& w_n = \Asig u_n + \B(u_n) - u_{n-1} \ &\mbox{ in }&
 \Esig'\label{td2}
\end{alignat}
for $n = 1,2,\ldots,N$.
In order to show existence of $(u_n,w_n)$ satisfying~\eqref{td1}-\eqref{td2}, 
let us introduce the functional $F_n : \Esig \to \mathbb R$
given by
$$
F_n(u) := \dfrac \tau 2 \left\| \dfrac{u-u_{n-1}}\tau \right\|_{\Xzs'}^2
+ \dfrac{C(\sigma)}4 \iint_{\mathbb R^{2N}}
\dfrac{|u(x)-u(y)|^2}{|x-y|^{N+2\sigma}} \; \d x \, \d y
+ \int_\Omega \hat \beta(u) \; \d x
- \int_\Omega u_{n-1} u \; \d x
$$
for $u \in \Esig$, where $\hbeta$ denotes the primitive function of
$\beta$ such that $\hat\beta(0)=0$. Then $F_n$ is strictly convex,
coercive and of class $C^1$ in $\Esig$. Therefore one can take a unique
minimizer $u_n \in \Esig$ of $F_n$. Let us recall that $\As^{-1} : \Xzs'
\to \Xzs$ is the inverse duality mapping, whence $\As^{-1}$ coincides
with the Fr\'echet derivative of the functional
$$
K(v) := \dfrac 1 2 \|v\|_{\Xzs'}^2 \quad \mbox{ for } \ v \in \Xzs'.
$$
Correspondingly, $K(\cdot)$ is of class $C^1$ in $\Xzs'$ ; moreover,
\begin{equation}\label{dK}
K'(v) = \As^{-1} v \quad \mbox{ for } \ v \in \Xzs'. 
\end{equation}
Indeed, for each $u \in \Esig$, one has
\begin{align*}
 K(u) - K(v)
 &\leq \left\langle u - v, \As^{-1} u \right\rangle_{\Xzs}
 \stackrel{\eqref{convert_pair}}{=} \left\langle \As^{-1} u, u - v
 \right\rangle_{\Esig}
 \quad \mbox{ for all } \ v \in \Esig,
\end{align*}
which implies \eqref{dK}. Thus we have
\begin{equation}\label{td-unif}
\As^{-1} \left( \dfrac{u_n-u_{n-1}}\tau\right) + \Asig u_n +
 \B(u_n) - u_{n-1} = 0 \ \mbox{ in } \Esig'.
\end{equation}
Setting
$$
w_n := 
- \As^{-1} \left(\dfrac{u_n-u_{n-1}}\tau\right) 
\in \Xzs
$$
and applying $\As$ to both sides, we obtain \eqref{td1} along with \eqref{td2}.

%
\subsection{A priori estimates}%
Test \eqref{td-unif} by $u_n \in \Esig$ $(\hookrightarrow \Xsz')$. Then
as in \eqref{dual-test} it follows that
\begin{align*}
\|u_n\|_{\Xsz'}^2-\|u_{n-1}\|_{\Xsz'}^2 
+ 2\tau \left( \|u_n\|_{\Xsigz}^2 + \|u_n\|_{L^p(\Omega)}^p \right)
&\leq 2 \tau (u_n, u_{n-1})_{L^2(\Omega)}\\
&\leq \tau \left( \|u_n\|_{L^2(\Omega)}^2 + \|u_{n-1}\|_{L^2(\Omega)}^2
 \right).
\end{align*}
Summing up, we deduce that
$$
\|u_n\|_{\Xsz'}^2
+ 2 \sum_{j = 1}^n \tau \left( \|u_j\|_{\Xsigz}^2 +
\|u_j\|_{L^p(\Omega)}^p \right)
\leq \|u_0\|_{\Xsz'}^2
+ 2 \sum_{j = 0}^n \tau \|u_j\|_{L^2(\Omega)}^2.
$$
Recall \eqref{ehrling-1} again to obtain
$$
\|u_n\|_{\Xsz'}^2
+ \sum_{j = 1}^n \tau \left( \|u_j\|_{\Xsigz}^2 +
2 \|u_j\|_{L^p(\Omega)}^p \right)
\leq C \|u_0\|_{L^2(\Omega)}^2 + C \sum_{j = 1}^n \tau \|u_j\|_{\Xsz'}^2.
$$
Hence, due to the discrete Gronwall inequality, we obtain 
\begin{equation}\label{td:e-0}
\max_n \|u_n\|_{\Xsz'}^2 + \sum_{n = 0}^N \tau \left( \|u_n\|_{\Xsigz}^2 +
\|u_n\|_{L^p(\Omega)}^p \right) \leq C.
\end{equation}

Test \eqref{td1} by $w_n \in \Xsz$ and \eqref{td2} by
$(u_n - u_{n-1})/\tau \in \Esig$ and employ \eqref{convert_pair} to get
\begin{equation}\label{ei-dc}
\|w_n\|_{\Xsz}^2
+ \dfrac{\|u_n\|_{\Xsigz}^2 - \|u_{n-1}\|_{\Xsigz}^2}{2\tau} 
+ \dfrac{\|u_n\|_{L^p(\Omega)}^p - \|u_{n-1}\|_{L^p(\Omega)}^p}{p\tau} 
\leq \dfrac{\|u_n\|_{L^2(\Omega)}^2 - \|u_{n-1}\|_{L^2(\Omega)}^2}{2\tau}.
\end{equation}
By summing up, we have
\begin{align}\label{ei-dc-2}
 \sum_{j=1}^n \tau \|w_j\|_{\Xzs}^2 + \dfrac 1 2 \|u_n\|_{\Xzsi}^2
      + \dfrac 1 p \|u_n\|_{L^p(\Omega)}^p
\leq \dfrac 1 2 \|u_0\|_{\Xzsi}^2 + \dfrac 1 p \|u_0\|_{L^p(\Omega)}^p
+ \dfrac 1 2 \|u_n\|_{L^2(\Omega)}^2.
\end{align}
Using \eqref{ehrling-1} and \eqref{td:e-0}, assumption~\eqref{eq:cond_iniz},
and the the boundedness of $\As : \Xsz \to \Xsz'$ along with \eqref{td1}, we
then deduce
\begin{equation}\label{est-dc}
\sum_{n = 1}^N \tau \|w_n\|_{\Xsz}^2 
+
\max_{n} \left(\|u_n\|_{\Xsigz}^2+\|u_n\|_{L^p(\Omega)}^p\right) \leq C, \quad
\sum_{n=1}^N \tau \left\| 
\dfrac{u_n - u_{n-1}} \tau \right\|_{\Xsz'}^2 \leq C.
\end{equation}
Let $\bar u_\tau$ and $\bar w_\tau$ be the piecewise constant
interpolants of $\{u_n\}$ and $\{w_n\}$, respectively, and let $u_\tau$
be the piecewise linear interpolant of $\{u_n\}$. More precisely, we
define $\bar u_\tau$, $u_\tau$ by
$$
  {\bar u_\tau (t) \equiv u_n },\quad
u_\tau (t) = \dfrac{t-t_{n-1}}{\tau} u_n + \dfrac{t_n-t}{\tau} u_{n-1} 
\quad \mbox{ for } \ t \in [t_{n-1},t_n),
$$
and $\bar w_\tau$ analogously.
Then they satisfy
\begin{alignat}{4}
 &\partial_t u_\tau + \As \bar w_\tau = 0 \ &\mbox{ in }&
 \Xsz',\label{td1i}\\
 &\bar w_\tau = \Asig \bar u_\tau + \B(\bar u_\tau) 
 - \bar u_\tau(\cdot - \tau) \ &\mbox{ in }&
 \Esig'.\label{td2i}
\end{alignat}
Moreover, the previous estimates \eqref{est-dc} 
can be rewritten in the form
\begin{equation}
\int^T_0 \|\bar w_\tau(t)\|_{\Xsz}^2 \d t 
+ \sup_{t \in [0,T]} \left( \|\bar u_\tau(t)\|_{\Xsigz}^2 
+ \|\bar u_\tau(t)\|_{L^p(\RN)}^p \right) \leq C,
\quad
\int^T_0 \|\partial_t u_\tau(t)\|_{\Xsz'}^2 \d t \leq C, \label{est}
\end{equation}
whence we easily get also 
\begin{equation}
\sup_{t \in [0,T]} \left( \|u_\tau(t)\|_{\Xsigz}^2 
+ \|u_\tau(t)\|_{L^p(\RN)}^p \right) \leq C.\label{est1}
\end{equation}
Note that
$$
 \langle \B(\bar u_\tau(t)), \phi \rangle_{L^p(\RN)}
 = \int_\Omega \beta(\bar u_\tau(x,t)) \phi(x) \; \d x
 \leq \|\bar u_\tau(t)\|_{L^p(\Omega)}^{p-1} \|\phi\|_{L^p(\mathbb R^N)}
\quad \mbox{ for all } \ \phi \in L^p(\mathbb R^N),
$$
which implies
\begin{align}
\sup_{t\in[0,T]}\|\B(\bar u_\tau(t))\|_{L^{p'}(\mathbb R^N)}
\leq \sup_{t\in[0,T]}\|\bar u_\tau(t)\|_{L^p(\Omega)}^{p-1} \leq C.
\label{e:Bu}
\end{align}

%
\subsection{\bf Convergence as $\tau \to 0$}%
From the estimates established so far, one can take a (non-relabeled)
subsequence of $\tau \to 0$ (equivalently, $N \to \infty$) such that
\begin{alignat}{4}
 \bar w_\tau &\to w \quad &&\mbox{ weakly in } L^2(0,T;\Xsz), \label{gs11}\\
 \As \bar w_\tau &\to \As w \quad &&\mbox{ weakly in } L^2(0,T;\Xzs'), \label{gs12}\\
 \bar u_\tau &\to u \quad &&\mbox{ weakly star in } L^\infty(0,T;\Esig), \label{gs13}\\
 \Asig \bar u_\tau &\to \Asig u \quad &&\mbox{ weakly in } L^2(0,T;\Xsigz'),\label{gs14}\\
 \B(\bar u_\tau(\cdot)) &\to \chi \quad &&\mbox{ weakly star in }
 L^\infty(0,T;L^{p'}(\mathbb R^N)), \label{gs15}\\
 u_\tau &\to u \quad &&\mbox{ weakly star in }
 L^\infty(0,T;\Esig),\label{gs16}\\
 \partial_t u_\tau &\to \partial_t u \quad &&\mbox{ weakly in } L^2(0,T;\Xsz'). \label{gs17}
\end{alignat}
Combining these facts with the compact embeddings $\Esig \hookrightarrow
\LO \hookrightarrow \Xsz'$, and using
the Aubin-Lions-Simon compactness lemma
(see~\cite[Theorem 5]{Simon}), one can verify that
\begin{equation}
u_\tau \to u \quad \mbox{ strongly in } C([0,T];\LO).\label{c:u2-C}
\end{equation}
Then $u$ belongs to $C([0,T];\LO) \cap C_w([0,T];\Esig)$ as well.
Moreover, we observe by \eqref{est-dc} that
\begin{align}
\|\bar u_\tau (t) - u_\tau (t)\|_{\Xsz'}
&= \left\| u_n - \left(\dfrac{t-t_{n-1}}{\tau} u_n + \dfrac{t_n-t}{\tau}
 u_{n-1} \right) \right\|_{\Xsz'}
\label{equi-Ju}\\
&= \dfrac{t_n-t}{\tau} \| u_n- u_{n-1}\|_{\Xsz'}
\stackrel{\eqref{est-dc}}{\leq} C \sqrt \tau
\quad \mbox{ for all } \ t \in [t_{n-1}, t_n),
\nonumber
\end{align}
which along with \eqref{c:u2-C} yields
$$
\sup_{t \in [0,T]} \| \bar u_\tau(t) - u(t)\|_{\Xsz'}
\leq \sup_{t \in [0,T]} \| \bar u_\tau(t) - u_\tau(t)\|_{\Xsz'}
+ \sup_{t \in [0,T]} \| u_\tau(t) - u(t)\|_{\Xsz'}
\to 0.
$$
Moreover, exploiting \eqref{ehrling-1} and \eqref{est}, for any $\vep > 0$, 
one can take $C_\vep \geq 0$ such that
\begin{align*}
 \|\bar u_\tau (t) - u(t) \|_{\LO}
 &\leq \vep \|\bar u_\tau (t) - u(t) \|_{\Xzsi}
 + C_\vep \| \bar u_\tau (t) - u(t) \|_{\Xsz'}\\
 &\leq \vep C + C_\vep \sup_{r \in [0,T]} \| \bar u_\tau (r) - u(r) \|_{\Xsz'} \quad \mbox{ for any } \ t \in [0,T],
\end{align*}
whence there follows that
\begin{alignat}{4}
\bar u_\tau &\to u \quad &&\mbox{ strongly in } L^\infty(0,T;\LO),
\label{c:bu2-i}\\
\bar u_\tau(t) &\to u(t) \quad &&\mbox{ strongly in } \LO \ \mbox{
 for all } \ t \in [0,T],
\label{c:bu(t)2}\\
&&&\mbox{ weakly in } \Esig \ \mbox{ for all } \ t \in [0,T].
\label{c:buE}
\end{alignat}
Furthermore, noting that
\begin{align*}
\lefteqn{
 \| \bar u_\tau(t-\tau) - u(t)\|_{\Xsz'}
}\\
&\leq
\| \bar u_\tau(t-\tau) - u_\tau(t-\tau)\|_{\Xsz'}
+
\| u_\tau(t-\tau) - u_\tau(t)\|_{\Xsz'}
+
\| u_\tau(t) - u(t)\|_{\Xsz'}\\
&\leq
C \sqrt \tau
+ \int^t_{t-\tau} \|\partial_r u_\tau(r)\|_{\Xsz'} \d r
+ \| u_\tau(t) - u(t)\|_{\Xsz'}
\end{align*}
for all $t \in [\tau, T]$, we also deduce
from \eqref{est} and \eqref{c:u2-C} that
$$
\sup_{t \in [\tau,T]}
\| \bar u_\tau(t-\tau) -  u(t)\|_{\Xsz'}
\leq C \sqrt \tau + \sup_{t \in [0,T]} \| u_\tau(t) -
  u(t)\|_{\Xsz'}
\to 0 \quad \mbox{ as } \ \tau \to 0.
$$
As in \eqref{c:bu2-i}, one can further obtain
\begin{alignat}{4}
\bar u_\tau(\cdot-\tau) &\to u \quad &&\mbox{ strongly in }
L^\infty(0,T;\LO), \label{gs21}\\
\bar u_\tau(t-\tau) &\to u(t) \quad &&\mbox{ strongly in }
\LO \ \mbox{ for all } \ t > 0. \label{gs22}
\end{alignat}
In particular, \eqref{gs22} also implies $w = \Asig u + \chi -
u$ in $\Esig'$. 

We next verify that $\chi = \B(u)$ by using Minty's trick. To
this end, we observe that
\begin{align*}
\langle \B(\bar u_\tau(t)) , \bar u_\tau(t) \rangle_{L^p(\mathbb R^N)}
= ( \bar w_\tau(t) , \bar u_\tau(t) )
 - \|\bar u_\tau(t)\|_{\Xsigz}^2 
 + ( \bar u_\tau(t-\tau) , \bar u_\tau (t) ).
\end{align*}
Hence, using \eqref{gs11}, \eqref{c:bu2-i} and \eqref{gs22}, 
we obtain
\begin{align}
\lefteqn{
\limsup_{\tau \to 0}\int^T_0 \langle \B(\bar u_\tau(t)) , \bar u_\tau(t)
 \rangle_{L^p(\mathbb R^N)} \; \d t
}\nonumber\\
&= \lim_{\tau \to 0} \int^T_0 ( \bar w_\tau(t) , \bar u_\tau(t)
 ) ~\d t 
 - \liminf_{\tau \to 0}\int^T_0 \|\bar u_\tau(t)\|_{\Xsigz}^2 ~\d t
\label{minty}\\
&\quad + \lim_{\tau \to 0}\int^T_0 ( \bar u_\tau(t-\tau) , \bar u_\tau (t)
 )~\d t\nonumber\\
&\leq \int^T_0 ( w(t) , u(t) )~\d t 
 - \int^T_0 \|u(t)\|_{\Xsigz}^2 ~\d t
 + \int^T_0 ( u(t) , u (t) ) ~\d t\nonumber\\
&= \int^T_0 \langle \chi(t), u(t) \rangle_{L^p(\mathbb R^N)} ~\d t,
\nonumber
\end{align}
which along with \eqref{gs13}, \eqref{gs15} and the maximal monotonicity of the operator $u \mapsto
\B(u)$ from $L^p(\mathbb R^N)$ to $L^{p'}(\mathbb R^N)$ yields $\chi = \B(u)$ 
in $L^{p'}(\mathbb R^N)$. Thus $(u,w)$ is a weak solution of
\eqref{eq:fCH}-\eqref{eq:bc}.

%
\subsection{Energy inequalities}%
In order to derive \eqref{enineq-0}, let us fix $t \in (0,T)$ at which
\eqref{eq:cp_weak} holds and define
the Yosida approximation $\beta_\vep : \mathbb R \to \mathbb R$ of
$\beta$, i.e., $\beta_\vep (r) := (r - j_\vep(r))/\vep =
\beta(j_\vep(r))$ for $r \in \mathbb R$, where $j_\vep$ stands
for the resolvent of $\beta$ defined by $j_\vep(r) := (1 + \vep
\beta)^{-1} (r)$.
Then, since $\beta_\vep$ is Lipschitz continuous
and $\beta_\vep(0)=0$, one can observe that
$\beta_\vep (u(\cdot)) \in \Esig$ if $u \in \Esig$.
Hence, we can test \eqref{eq:cp_weak} by
$\beta_\vep(u(\cdot,t))$ to get
$$
\langle \Asig u(t), \beta_\vep(u(\cdot,t)) \rangle_{\Xzsi}
+ \langle \B(u(t)), \beta_\vep(u(\cdot,t)) \rangle_{\Esig}
= \left(w(t) + u(t), \beta_\vep(u(\cdot,t))\right).
$$
Here we note as in \eqref{eq:monotone:e} that
$$
\langle \Asig u(t), \beta_\vep(u(\cdot,t)) \rangle_{\Xzsi} \geq 0.
$$
Moreover, by the definition of Yosida approximation and the monotonicity
of $\beta$, we infer that
$$
\langle \B(u(t)), \beta_\vep(u(\cdot,t)) \rangle_{\Esig}
= \int_\Omega \beta(u(x,t)) \beta_\vep(u(x,t)) \; \d x
\geq \|\beta_\vep(u(\cdot,t))\|_{L^2(\Omega)}^2
= \|\beta_\vep(u(\cdot,t))\|_{\LO}^2.
$$
Thus we obtain
\begin{equation}\label{gs31}
 \|\beta_\vep(u(\cdot,t))\|_{\LO}
  \leq \|w(t)\|_{\LO} + \|u(t)\|_{\LO}
  \quad \mbox{ for a.e. } t \in (0,T),
\end{equation}
which implies
$$
\beta_\vep(u(\cdot,t)) \to b_t \quad \mbox{ weakly in } \LO
\mbox{ as } \vep \to 0
$$
for some $b_t \in \LO $. On the other hand, let us recall that
$\beta_\vep(r) = \beta(j_\vep(r))$. Moreover, since $j_\vep$ is non-expansive
(i.e., Lipschitz continuous with the Lipschitz constant $1$) and
$j_\vep(0) = 0$, one can easily check that
$$
\|j_\vep(u(\cdot,t))\|_{\Esig} \leq \|u(t)\|_{\Esig},
$$
which yields 
$$
j_\vep(u(\cdot,t)) \to u(t) \quad \mbox{ weakly in } \Esig \mbox{ and
strongly in } \LO,
$$
as $\vep \to 0$.
Here we also used the fact that
$$
|j_\vep(u(x,t))-u(x,t)| \leq \vep |\beta_\vep(u(x,t))| \leq \vep
|\beta(u(x,t))| \quad \mbox{ for a.e. } x \in \Omega.
$$
Therefore, by virtue of the demiclosedness of the maximal monotone
operator $u \mapsto \beta(u(\cdot))$ in $\LO \times \LO$, we conclude that
$b_t = \beta(u(\cdot,t))$ a.e.~in $\Omega$.
Moreover, \eqref{enineq-0} follows from the weak
lower-semicontinuity of the norm $\|\cdot\|_{\LO}$ in
$\LO$. Furthermore, integrating \eqref{enineq-0} in time, we obtain 
$\beta(u) \in L^2(0,T;\LO)$, as $u$ and $w$ belong to $L^2(0,T;\LO)$.

We shall finally prove that $t \mapsto \Esi(u(t))$ is differentiable
a.e.~in $(0,T)$ and derive the energy inequality, namely
\begin{equation}\label{ei1}
 \|w(t)\|_{\Xsz}^2 + \dfrac{\d}{\d t} \Esi(u(t)) \leq 0
  \quad \mbox{ for a.e. } t \in (0,T),
\end{equation}
which can be also rewritten as
\begin{equation}\label{ei2}
\langle \partial_t u(t), w(t) \rangle_{\Xsz}
\geq \dfrac{\d}{\d t} \Esi(u(t)) \quad \mbox{ for a.e. } t \in (0,T).
\end{equation}
Moreover, the right-continuity of the function $t \mapsto u(t)$
in the strong topology of $\Esig$ will also
follow as a by-product of our argument.
\begin{remark}\label{R:enineq}
 {\rm
 Before proceeding with a proof, it is worth stressing
 that, differently from what happens in the non-fractional case,
 the differentiability of  $\Esi(u(t))$ and the energy 
 inequality \eqref{ei1} are not straightforward. 
 Indeed, the energy functional $\Esi$ is smooth
 in $\Esig$ but non-convex. Hence, if one attempts to apply a standard
 chain-rule to $\Esi$ and $u(t)$, the differentiability of $u(t)$ in the
 strong topology of $\Esig$ is needed. However, $t \mapsto
 u(t)$ turns out to be differentiable only in the weaker space
 $\Xzs'$. When dealing with the standard Cahn-Hilliard equation
 (i.e., for $s = \sigma = 1$), 
 this problem may be overcome by rewriting the energy
 functional corresponding to $\Esi$ as the sum of a convex and of a
 concave part and by applying a generalized chain-rule for convex but
 (possibly) non-smooth functionals (see, e.g.,~\cite{brezis73}). 
 However, in the present case, 
 this kind of procedure seems to
 work only when $\sigma\ge s$. We shall give the highlights
 of a proof of this fact in Subsec.~\ref{sub:eneq} below.
 }
\end{remark}

In order to show \eqref{ei1}, we start with noting that, from
\eqref{ei-dc-2} and interpolation, there follows
$$
\int^t_0 \|\bar w_\tau(r)\|_{\Xsz}^2 \d r
+ \Csi(\bar u_\tau(t)) 
- \dfrac 1 2 \|\bar u_\tau(t)\|_{L^2(\Omega)}^2 
\leq \Csi(u_0) - \dfrac 1 2 \|u_0\|_{L^2(\Omega)}^2
\ \mbox{ for } \ 0 \leq t \leq T,
$$
where $\Csi(\cdot)$ denotes the convex functional of class
$C^1$ on $\Esig$ given by
$$
\Csi(v) = \dfrac 1 2 \|v\|_{\Xsigz}^2 + \dfrac 1 p
\|v\|_{L^p(\Omega)}^p
\quad \mbox{ for } \ v \in \Esig.
$$
Using the convergence relations obtained so far and the weak lower
semicontinuity of $\Csi(\cdot)$ in $\Esig$, we deduce that
\begin{equation}\label{ei-i0}
\int^t_0 \|w(r)\|_{\Xsz}^2 \d r
+ \Esi(u(t)) \leq \Esi(u(0))
\quad \mbox{ for all } \ t \in [0,T].
\end{equation}
From the uniqueness of the solution and the fact that $u(t) \in \Esig$ for
all $t \in [0,T]$, one can also derive
\begin{equation}\label{ei-i}
\int^\tau_t \|w(r)\|_{\Xsz}^2 \d r
+ \Esi(u(\tau)) - \Esi(u(t)) \le 0
\quad \mbox{ for all } \ 0 \leq t \leq \tau \leq T,
\end{equation}
which also implies that $\Esi(u(\cdot))$ is nonincreasing on $[0,T]$, whence 
it is differentiable a.e.~in $(0,T)$.
Since $u \in C([0,T];\LO) \cap C_w([0,T];\Esig)$ and
$\Csi(\cdot)$ is weakly lower semicontinuous in $\Esig$,
$\Esi(u(\cdot))$ is right-continuous on $[0,T)$, i.e., $\Esi(u(\tau)) \to
\Esi(u(t))$ as $\tau \searrow t$. Then, the same property
holds for $\|u(\cdot)\|_{\Xsigz}^2$ and
$\|u(\cdot)\|^p_{L^p(\RN)}$ by the weak lower semicontinuity of the
norms. Therefore due to the uniform convexity
of $\Xsigz$ and $L^p(\RN)$, we can also verify that $u$
is right-continuous on $[0,T)$ in the strong topology of
$\Esig$.

Furthermore, let $t$ belong to the set
$$
  \mathcal I := \big\{ t \in [0,T] \colon
   \mbox{$\Esi(u(\cdot))$ is differentiable at $t$, 
    and $t$ is a Lebesgue point of $\|w(\cdot)\|_{\Xsz}^2$} \big\}.
$$
Then $(0,T) \setminus \mathcal I$ has zero Lebesgue measure.
Dividing both sides of \eqref{ei-i} by $\tau - t > 0$
and passing to the limit as $\tau \searrow t$, we obtain \eqref{ei1}.


\subsection{Energy equality}
\label{sub:eneq}
We prove here that, under the condition $\sigma\ge s$,
$u$ belongs to $C([0,T];\Esig)$, the energy 
$\Esi(u(t))$ is absolutely continuous on 
$[0,T]$, and the inequality \eqref{energy_inequality} 
can be replaced by the following {\sl energy identity}\/:
\begin{equation}\label{energy_identity}
   \|w(t)\|_{\Xsz}^2 + \dfrac{\d}{\d t} \Esi(u(t)) = 0
    \quad \mbox{for a.e.~} t \in (0,T).
\end{equation}
The key tool in order to get \eqref{energy_identity}
is the following chain-rule formula,
which can be proved by adapting the argument given 
in~\cite[Lemma 4.1]{RS}:
\begin{lemma}\label{lemma:c_r}
Let $(\calV, \calH, \calV')$ be a Hilbert 
triple and let $\Psi:\calH\to (-\infty,+\infty]$
be a convex, proper and lower semicontinuous
functional. Moreover, let us assume that, for some $k_1>0$,
$k_2\ge 0$, there holds
\begin{equation}\label{coerc_psi}
  \Psi(v)\ge k_1 \| v \|_{\calH}^2 - k_2 \,\,\,\,\forall\, v\in \calH.
\end{equation}
Denote with $\mathcal A$ the
subdifferential of $\Psi$ with respect to 
the scalar product of $\calH$, and consider, for $T>0$,
$v\in W^{1,2} (0,T;\calV')\cap L^2(0,T;\calV)$ and
$\eta\in L^2(0,T;\calV)$ with $\eta(t)\in \mathcal{A}v(t)$
for a.e.~$t\in (0,T)$. Then, the function 
$t\mapsto \Psi(v(t))$ is absolutely continuous in $[0,T]$.
Moreover,
\begin{equation}\label{eq:c_r}
  \int_{r}^t\langle \partial_t v(\tau),\eta(\tau)\rangle_{\calV} \,\d \tau
   = \Psi(v(t))-\Psi(v(r))
    \quad \mbox{for all } \ 0 \leq r \leq t \leq T.
\end{equation} 
\end{lemma}
\noindent%
We apply the above Lemma with the following 
choices:
\begin{itemize}
\item $\calH=H_0$, $\calV=\Xsz$, see \eqref{H-tr}.
\item $\Psi(v) = \frac12\| v\|^2_{\Xsigz} 
 + \int_{\Omega}\hat{\beta}(v(x))\,\d x$.
\end{itemize}
Then, clearly, $\Psi$ is proper, lower semicontinuous
and convex. Moreover, thanks to the fractional Poincar\'e
inequality \eqref{eq:poincare},
it satisfies the coercivity assumption \eqref{coerc_psi}. 
Now, let $u,w$ be the solution given by Theorem~\ref{th:eu}. 
Then, by Definition~\ref{def:weaksol}, we have
$$ 
  u\in W^{1,2} (0,T;\calV') \,\,\,\,\hbox{ and } \ w\in L^2(0,T; \calV).
$$
On the other hand, being $\sigma\ge s$, we also have 
$$ 
  u \in L^2(0,T;\calV).
$$
Thus, setting $\eta := w + u$, it follows that $\eta\in L^2(0,T;\calV)$;
moreover, thanks to equation~\eqref{eq:cp_weak},
$\eta(t)\in \partial\Psi(u(t))$ for a.e.~$t\in (0,T)$.
Hence, \eqref{energy_identity} follows from
Lemma~\ref{lemma:c_r}.
\begin{remark}\label{new:sig-s}
{\rm
 When $\sigma < s$, 
 we do not know whether
 $\eta = w + u \in L^2(0,T;\calV)=L^2(0,T;\Xsz)$,
 because $\Xsigz$ is not included into $\Xsz$.
 Hence, the assumptions of Lemma~\ref{lemma:c_r}
 are not necessarily satisfied and the energy 
 identity~\eqref{energy_identity} remains
 as an open issue in this case.
}
\end{remark}
%

%
%

\section{Singular limits}
\label{sec:conve}

\bigskip

\subsection{Uniform estimates}\label{Ss:UE}

Before proving our results on singular limits as $\sigma \to 0$ or $s \to
0$, we first establish uniform estimates
with respect to $s$ and $\sigma$ for the unique weak solution
$(u,w)$ to the fractional Cahn-Hilliard system
\eqref{eq:fCH}-\eqref{eq:bc}. 
To this end, let us recall the inequality \eqref{enineq-0} of
Theorem \ref{th:eu},
\begin{equation}\label{enineq-0-i}
\|\beta(u(\cdot,t))\|_{\LO}^2 \leq 2 \left( \|w(t)\|_{\LO}^2
+ \|u(t)\|_{\LO}^2 \right)
\quad \mbox{ for a.e. } \ t \in (0,T).
\end{equation}
Moreover, we also recall the energy inequality \eqref{energy_inequality},
$$
\dfrac{\d}{\d t} \Esi(u(t)) +  \|w(t)\|_{\Xsz}^2 \leq 0
  \quad \mbox{ for a.e. } t \in (0,T).
$$
Then, integrating both sides in $t$ we get
\begin{equation}\label{eq:energy1}
 \sup_{t \in [0,T]} \left(
 \dfrac 1 2 \|u(t)\|_{\Xzsi}^2 + \int_{\Omega}W(u(x,t))\, \dx \right)
   + \|w\|_{L^2(0,T;\Xzs)}^2 \le Q(\Ezsi).
\end{equation}
Here and henceforth, $Q(\cdot)$ denotes a computable nonnegative-valued
function which is monotonely increasing in its argument(s) and may vary
from line to line. In particular, the expression of $Q$ may depend on
$p$, $T$ and $|\Omega|$; however, it is always independent both of
$\sigma$ and of $s$.

{\bf In case $p>2$\,:} since $W$ is coercive, it follows immediately
from \eqref{eq:energy1} that
\begin{equation}\label{eq:estimate_uw}
  \| u \|_{L^\infty(0,T;\Xzsi)}^2 
   + \| u \|_{L^\infty(0,T;L^p(\Omega))}^p
   + \|w\|_{L^2(0,T;\Xzs)}^2 \le Q(\Ezsi),
\end{equation}
which along with the boundedness of $\As : \Xzs \to \Xzs'$ and a
comparison of terms in \eqref{eq:ch_weak} gives
\begin{equation}\label{eq:energy2}
\| \partial_t u \|_{L^2(0,T;\Xzs')}^2
\stackrel{\eqref{eq:ch_weak}}{=} \| \As w \|_{L^2(0,T;\Xzs')}^2 
\le Q(\Ezsi).
\end{equation}
Moreover, by \eqref{e:Bu},
\begin{equation}
\|\B(u)\|_{L^\infty(0,T; L^{p'}(\RN))} 
\leq \|u\|_{L^\infty(0,T;L^p(\Omega))}^{p-1} \leq Q(\Ezsi). 
\end{equation}
On the other hand, combining \eqref{eq:estimate_uw} with
\eqref{enineq-0-i}, we get
\begin{equation}\label{est:beta}
  \|\beta(u)\|_{L^2(0,T;\LO)}^2 \le Q(\Ezsi).
\end{equation}
In turn, this estimate clearly implies that
\begin{equation}\label{eq:estW3}
  \|W'(u)\|_{L^2(0,T;\LO)}^2 \le Q(\Ezsi).
\end{equation}
Hence by \eqref{eq:cp_weak}, we find that
\begin{align*}
\langle \Asig u(t), \phi \rangle_{\Esig}
&\stackrel{\eqref{eq:cp_weak}}= \langle w(t) - \B(u(t)) + u(t), \phi \rangle_{\Esig}\\
&= \int_\Omega \left( w(x,t) - \beta(u(x,t)) + u(x,t) \right) \phi(x) \; \d x\\
&\leq \left( \|w(t)\|_{\LO} + \|\beta(u(\cdot,t))\|_{\LO} +
 \|u(t)\|_{\LO}\right) \|\phi\|_{\LO}
\quad \mbox{ for all } \ \phi \in \DO.
\end{align*}
Since $\DO$ is dense in $\LO$, one has 
\begin{equation}\label{est:Delsig}
 \left\|\Asig u\right\|_{L^2(0,T;\LO)} \le Q(\Ezsi),
\end{equation}
where $\Asig : \LO \to \LO$ stands for the
$\LO$-fractional Laplacian with domain $D(\Asig) \subsetneq \LO$ (see \S
\ref{Ss:L2}) and $\Asig u(t) : \LO \to \mathbb R$ is the unique bounded
linear extension onto $\LO$ of the functional $\Asig u(t) : \Xsigz \to
\mathbb R$.

{\bf In case $ 1 < p < 2$ and $\sigma \in (0,1)$ is fixed:} (the argument
below is still available for $p > 2$ as well) by applying $\As^{-1}$ to
both sides of \eqref{eq:ch_weak} and
by utilizing \eqref{eq:cp_weak}, we have
$$
\As^{-1} \left( \partial_t u (t) \right) + \Asig u(t) + \B(u(t)) -
u(t) = 0
\ \mbox{ in } \Esig', \quad 0 < t < T.
$$
Test it by $u(t) \in \Esig$. It follows that
$$
\dfrac 1 2 \dfrac{\d}{\d t} \|u(t)\|_{\Xsz'}^2
+ \|u(t)\|_{\Xsigz}^2 + \|u(t)\|_{L^p(\Omega)}^p =
\|u(t)\|_{L^2(\Omega)}^2
\quad \mbox{ for a.e. } \ 0 < t < T.
$$

Set $X = \Xzr$ or $X = \mathcal E_r := \Xzr \cap L^p(\RN)$ for a fixed
constant $r \in (s,1)$ or $X = H^1_0(\Omega)$. Then $X$ is continuously
embedded in $\Xsz$ uniformly for $s \to 0$. More precisely, there exists
a constant $C_0 > 0$ independent of $s \to 0$ such that 
\begin{equation}\label{embed-X}
 \|v\|_{\Xzs} \leq C_0 \|v\|_X \quad \mbox{ for all } \ v \in X.
\end{equation}
Indeed, as in~\cite[Proof of Proposition 2.1]{Dine_Pala_Vald}, one can
verify that, for all $v \in X$,
\begin{align*}
\lefteqn{
\dfrac{C(s)}2
\iint_{\mathbb R^{2N}} \dfrac{|v(x)-v(y)|^2}{|x-y|^{N+2s}} ~ \d x \; \d y
}\\
&= \dfrac{C(s)}2 \int_{\mathbb R^N} \int_{\{x \in \mathbb R^N \colon
 |x-y| > 1\}} \dfrac{|v(x)-v(y)|^2}{|x-y|^{N+2s}} ~ \d x \; \d y
\\
&\quad 
 + \dfrac{C(s)}2 \int_{\mathbb R^N} \int_{\{x \in \mathbb R^N \colon
 |x-y| \le 1\}} \dfrac{|v(x)-v(y)|^2}{|x-y|^{N+2s}} ~ \d x \; \d y\\
&\leq \dfrac{C(s)}s |\mathbb S^{N-1}| \|v\|_{L^2(\Omega)}^2
+
 \begin{cases}
 \frac{C(s)}{C(r)}\|v\|_{\Xzr}^2
  &\mbox{ for } \ X = \Xzr \mbox{ or } \mathcal E_r,\\[2mm]
  \frac{|\mathbb S^{N-1}|}{4} \frac{C(s)}{1-s} \|\nabla v\|_{L^2(\Omega)}^2
  &\mbox{ for } \ X = H^1_0(\Omega).
 \end{cases}
\end{align*}
Here $|\mathbb S^{N-1}|$ stands for the surface area of a unit sphere in
$\mathbb R^N$. 
Finally, exploit the asymptotics \eqref{eq:limC} of $C(r)$ as $r \searrow 0$ to
obtain \eqref{embed-X}. Moreover,  \eqref{embed-X} yields
\begin{equation}\label{embed-X'}
 \|\zeta\|_{X'} \leq C_0 \|\zeta\|_{\Xzs'} \quad \mbox{ for all } \
  \zeta \in \Xzs',
\end{equation}
which particularly gives
\begin{equation}\label{embed-X'2}
 \|u\|_{X'} \leq C_0 \|u\|_{\Xzs'} \quad \mbox{ for all } \ u \in \LO
 \, (\simeq \LO').
\end{equation}
Indeed, for any $\phi \in X \subset \Xzs$, one finds that
$$
\langle \zeta, \phi \rangle_X
= \langle \zeta, \phi \rangle_{\Xzs}
\leq \|\zeta\|_{\Xzs'} \|\phi\|_{\Xzs}
\leq C_0 \|\zeta\|_{\Xzs'} \|\phi\|_X
\quad \mbox{ for } \ \zeta \in \Xzs',
$$
which gives \eqref{embed-X'}.
On the other hand, from the dense and compact embeddings $\Xsigz
\hookrightarrow \LO \, (\simeq \LO') \hookrightarrow X'$ along
with Ehrling's compactness lemma~\cite[Lemma 8]{Simon}, for any $\vep > 0$
there exists a positive constant $C_{\vep,\sigma}$, which is independent of $s$
but may depend on $\sigma$, such that
\begin{equation}\label{ehrling-2}
\|v\|_{L^2(\Omega)}^2 = \|v\|_{\LO}^2 \leq \vep \|v\|_{\Xsigz}^2 +
 C_{\vep,\sigma} \|v\|_{X'}^2 \quad \mbox{ for all } \ v \in \Xsigz.
\end{equation}
Therefore we deduce that
$$
\dfrac 1 2 \dfrac{\d}{\d t} \|u(t)\|_{\Xsz'}^2
+ \dfrac 1 2 \|u(t)\|_{\Xsigz}^2 + \|u(t)\|_{L^p(\Omega)}^p
\leq C_{\frac 1 2,\sigma} \|u(t)\|_{X'}^2
\quad \mbox{ for a.e. } \ 0 < t < T.
$$
The integration of both sides over $(0,t)$ along with \eqref{embed-X'2}
yields
$$
\|u(t)\|_{X'}^2 + \int^t_0 \left( \|u(\tau)\|_{\Xsigz}^2
+ \|u(\tau)\|_{L^p(\Omega)}^p \right) \d \tau
\leq C_\sigma \left( \|u_0\|_{\Xsz'}^2 + \int^t_0 
\|u(\tau)\|_{X'}^2 \, \d \tau \right)
$$
for some constant $C_\sigma > 0$ (depending on $\sigma$).
Hence, exploiting Gronwall's inequality, we obtain
\begin{equation}\label{eq:energy3}
 \sup_{t \in [0,T]} \|u(t)\|_{X'}^2
+ \int^T_0 \left( \|u(t)\|_{\Xsigz}^2
+ \|u(t)\|_{L^p(\Omega)}^p \right) \d t
\leq Q(C_\sigma,\|u_0\|_{\Xsz'}^2).
\end{equation}
Apply \eqref{ehrling-2} to \eqref{eq:energy1} and employ
\eqref{eq:energy3}. Then we obtain \eqref{eq:estimate_uw} with a bound
depending on $C_\sigma$, $\|u_0\|_{\Xzs'}$ and $\Ezsi$. 
Furthermore, relations analogous to
\eqref{eq:energy2}-\eqref{est:Delsig} also follow with similar
bounds. More precisely, one deduces that
\begin{align}
\|u\|_{L^\infty(0,T;\Xzsi)}^2 + \|u\|^p_{L^\infty(0,T;L^p(\Omega))}
 + \|\beta(u)\|_{L^2(0,T;\LO)}^2
 + \|W'(u)\|_{L^2(0,T;\LO)}^2 
\qquad
\label{unif_bound-2}\\
 + \left\|\Asig u\right\|_{L^\infty(0,T;\Xzsi')}^2
 \le Q(C_\sigma,\|u_0\|_{\Xsz'}^2, \Ezsi).
\nonumber
\end{align}
Moreover, we also have
$$
\|\partial_t u\|_{L^2(0,T;\Xzs')}^2 
+ \|\As w\|_{L^2(0,T;\Xzs')}^2
+ \|w\|_{L^2(0,T;\Xzs)}^2
\leq Q(C_\sigma,\|u_0\|_{\Xsz'}^2, \Ezsi).
$$
Hence by virtue of \eqref{embed-X'} and Poincar\'e's
inequality \eqref{eq:poincare-opt} along with Proposition
\ref{prop:asy_1eigen}, it follows that
\begin{equation}\label{unif_bound-2.5}
\| \partial_t u \|_{L^2(0,T;X')}^2 + \|\As w\|_{L^2(0,T;X')}^2 +
 \|w\|_{L^2(0,T;\LO)}^2 \le Q(C_\sigma,\|u_0\|_{\Xsz'}^2, \Ezsi).
\end{equation}
Finally, by \eqref{e:Bu}, it holds that
\begin{equation}\label{unif_bound-3}
 \|\B(u)\|_{L^\infty(0,T;L^{p'}(\mathbb R^N))} \leq Q(C_\sigma,\|u_0\|_{\Xsz'}^2,
  \Ezsi).
\end{equation}


\subsection{Limit of fractional Laplacian in Bochner spaces}

In this section, we shall generalize Lemma \ref{cor:weak_limit_s_zero}
for later use of proving the convergence of $\Ark u_k$ as $r_k \searrow 0$ in an
appropriate Bochner space.
Throughout this subsection, we use the notation $\Xb$ and $\Xzb$ even
for $\beta \geq 1$ in the following sense
\begin{align*}
 \Xb &:= \widehat{H^\beta}(\RN) = \left\{
u \in \mathcal S(\RN)' \colon \left(1 + |\xi|^2 \right)^{\beta/2} \hat u(\xi) \in L^2(\RN_\xi)
\right\},\\
 \Xzb &:= \left\{ u \in \Xb \colon u = 0 \mbox{ in } \RN \setminus
 \Omega \right\} \quad \mbox{ for } \ \beta \geq 1.
\end{align*}
Then one finds that
$$
\Xb \hookrightarrow \mathcal{X}_{\gamma} \hookrightarrow \mathcal{X}_0 =
\LR, \quad
\Xzb \hookrightarrow \mathcal{X}_{\gamma,0} \hookrightarrow
\mathcal{X}_{0,0} = \LO \quad
\mbox{ if } \ \beta \geq \gamma > 0
$$
with continuous densely defined canonical injections. Hence we also
have dual relations, 
$$
\LR' \hookrightarrow
\mathcal{X}_{\gamma}' \hookrightarrow \Xb',
\quad
\LO' \hookrightarrow
\mathcal{X}_{\gamma,0}' \hookrightarrow \Xzb'
\quad \mbox{ if } \ \beta \geq \gamma > 0
$$
densely and continuously.
%
%
For each $u \in \Xzb$ and $\beta, \gamma \geq 0$, one can define $T(u)
\in \mathcal X_\gamma'$ by
$$
\langle T(u), \phi \rangle_{\mathcal X_\gamma} 
:= \int_\Omega u(x) \phi(x) \, \d x \quad \mbox{ for } \ \phi \in
\mathcal X_\gamma.
$$
Then $T : \Xzb \to \mathcal X_\gamma'$ is continuous due to the
continuous embeddings described above. Hence $\Xzb$ is continuously
embedded in $\mathcal X_{\gamma}'$ by $T$. From now on, we simply write $u$
instead of $T(u)$ if no confusion may arise.

\begin{lemma}\label{cor:weak_limit_s_zero-rev}
 Let $u$ and $\xi$ be integrable functions 
  of $(0,T)$
 with values in $\Xzb'$ for some constant $\beta > 0$ 
 satisfying $\beta \neq n - 1/2$ with $n \in \mathbb N$.
 Let $\{r_k\}$ be a sequence in $(0,\beta)$ such that $r_k \searrow
 0$ as $k \to \infty$ and consider a sequence $\{u_k\}$ of strongly
 measurable functions in $(0,T)$ with values in $\Xzrk$.
In addition,
 assume that
\begin{align}
 \int^T_0 \left\langle u_k(t) , \varphi_k \right\rangle_{\Xb} \phi(t) \;
 \d t 
&\to  \int^T_0 \left\langle u(t) , \varphi \right\rangle_{\Xzb} \phi(t) \;
 \d t ,\label{lim_Dels:H1}\\
 \int^T_0 \left\langle \Ark u_k(t), \varphi
 \right\rangle_{\Xzrk} \phi(t) \; \d t
 &\to  \int^T_0 \left\langle \xi(t), \varphi
 \right\rangle_{\Xzb} \phi(t) \; \d t
\label{lim_Dels:H2}
\end{align}
for any $\varphi \in \DO$, $\phi \in
 C^\infty_0(0,T)$ and $\varphi_k \in \Xb$ satisfying
 $\varphi_k \to \varphi$ strongly in $\Xb$.
 Then it holds that $u (t) = \xi(t)$ in $\Xzb'$ for a.e.~$t
 \in (0,T)$.
\end{lemma}
\begin{proof}
From the continuous embeddings $\Xzrk \hookrightarrow \LR \simeq
 (\LR)' \hookrightarrow \Xb'$, we note that the function $u_k$ is
 strongly measurable with values in $\Xb'$ as well. For any $\varphi \in \DO$
 and $\phi \in C^\infty_0(0,T)$, we observe by \eqref{lim_Dels:H1} that
\begin{align}
\lefteqn{
 \int^T_0 \left\langle \Ark u_k(t), \varphi \right\rangle_{\Xzrk}
 \phi(t) \; \d t
}\label{Apdx:eq1}\\
&= \dfrac{C(r_k)}2 \int^T_0
 \left(
 \iint_{\RRR \mathbb R^{2N} \EEE} 
 \dfrac{\left(u_k(x,t)-u_k(y,t)\right)
 \left(\varphi(x)-\varphi(y)\right)}{|x-y|^{N+2r_k}}
 \; \d x \, \d y \right) \phi(t) \; \d t\nonumber\\
&= \int^T_0 \left( \Drk \varphi , u_k(t) \right)
 \phi(t) \; \d t\nonumber\\
&= \int^T_0 \left\langle u_k(t), \Drk \varphi
 \right\rangle_{\Xb} \phi(t) \; \d t
\stackrel{\eqref{lim_Dels:H1}}\to 
\int^T_0 \left\langle 
u(t), \varphi  \right\rangle_{\Xzb}
 \phi(t) \; \d t. \nonumber
\end{align}
 Here we used the fact that
 $\Drk \varphi \in \Xb$ and $\Drk \varphi
 \to \varphi$ strongly in $\Xb$
 (uniformly in $t$) as $k \to \infty$ by Lemma \ref{lem:limit_s_zero}.
By virtue of \eqref{lim_Dels:H2}, the left-hand side of \eqref{Apdx:eq1}
 converges as follows:
$$
 \int^T_0 \left\langle \Ark u_k(t), \varphi \right\rangle_{\Xzrk}
 \phi(t) \; \d t
\stackrel{\eqref{lim_Dels:H2}}\to  \int^T_0 \left\langle \xi(t), \varphi \right\rangle_{\Xzb}
 \phi(t) \; \d t.
$$
Thus we obtain
$$
\int^T_0 \left\langle \xi(t), \varphi \right\rangle_{\Xzb}
 \phi(t) \; \d t
= 
\int^T_0 \left\langle u(t), \varphi \right\rangle_{\Xzb}
 \phi(t) \; \d t.
$$
Recall that $\DO$ is dense in $\Xzb$ if $\beta \neq
 n - 1/2$ for $n \in \mathbb N$ (see~\cite[Theorems 11.4 and 11.1,
 Chap.~I]{lions_mag}). Hence, from the arbitrariness of $\varphi \in
 \DO$, we have
$$
\int^T_0 \left( \xi(t) - u(t) \right) \phi(t) \; \d t = 0
\ \mbox{ in } \Xzb'
\quad \mbox{ for all } \ \phi \in C^\infty_0(0,T).
$$
Finally, applying du Bois-Reymond' lemma for Bochner integrals, we
 conclude that $\xi(t) = u (t)$ in $\Xzb'$ for a.e.~$t \in (0,T)$.
\end{proof}

\begin{remark}\label{Apdx:R:appli}
{\rm
\begin{enumerate}
 \item[(i)] All the assumptions of Lemma \ref{cor:weak_limit_s_zero-rev} can
	    be proved to hold whenever $u_k \to u$ weakly in
	    $L^1(0,T;\Xb')$ and $\Ark u_k(\cdot) \to \xi$ weakly in
	    $L^1(0,T;\Xzb')$ as $k \to \infty$ for some $u \in
	    L^1(0,T;\Xb')$ and $\xi \in L^1(0,T;\Xzb')$. Indeed, the product of
	    test functions $\varphi_k \phi$ converges to $\varphi \phi$
	    strongly in $L^\infty(0,T;\Xb)$ and $\varphi \phi$ belongs
	    to $L^\infty(0,T;\Xzb)$, since 
	    $\varphi_k \to \varphi$ strongly in $\Xb$ with $\varphi \in
	    \DO \subset \Xzb$ and $\phi \in C^\infty_0(0,T)$. Moreover,
	    since $\Xzb \subset \Xb$, one 
	    observes that $\Xb' \hookrightarrow \Xzb'$, whence
	    $u$ belongs to $L^1(0,T;\Xzb')$.
 \item[(ii)] One can also derive a similar result for sequences
	     independent of $t$. More precisely, if $u_k \in
	     \Xzrk$, $u_k \to u$ weakly in $\Xb'$, $\Ark u_k
	     \in \Xzb'$ and $\Ark u_k \to \xi$ weakly in
	     $\Xzb'$, then $\xi = u$ in $\Xzb'$. Indeed, set $u_k
	     (\cdot) \equiv u_k$. Then $u_k \to v(\cdot) \equiv u$
	     weakly in $L^p(0,T;\Xb')$ and $\Ark u_k(\cdot)
	     \equiv \Ark u_k$ converges to $\eta(\cdot)
	     \equiv \xi$ weakly in $L^p(0,T;\Xzb')$ for any $p \in
	     [1,\infty)$; hence, all the assumptions 
	     of Lemma~\ref{cor:weak_limit_s_zero-rev}
	     hold true by (i) above. 
\end{enumerate}
}
\end{remark}

\subsection{Proof of Theorem \ref{th:pm}}\label{Ss:PM}
Let $\{\sigma_k\}$ be a sequence in $(0,1)$ such that $\sigma_k \searrow 0$ and
let $(u_k,w_k)$ be the family of weak solutions to
\begin{alignat}{4} \label{eq:ch_weak_k}
 & \partial_t u_k + \As w_k = 0 \ &\hbox{ in }& \Xzs', \\
 & w_k = \Asigk u_k + \B(u_k) - u_k \ &\hbox{ in }&
 \Esik' \label{eq:cp_weak_k}
\end{alignat}
with $u_k(0)=u_{0,k}$. Then recalling uniform estimates
\eqref{eq:estimate_uw}-\eqref{est:Delsig} in \S \ref{Ss:UE} along with
hypothesis \eqref{hp:uzk}, one can
take weak limits $u\in L^\infty(0,T;\LpO) \cap W^{1,2}(0,T;\Xzs')$,
$w\in L^2(0,T;\Xzs)$, $\bar\beta \in L^2(0,T;\LO)$ and $\xi \in
L^2(0,T;\LO)$ such that, up to a (non-relabeled) subsequence of $\{k\}$,
\begin{alignat}{4} \label{eq:convu1}
 u_k &\to u \quad && \hbox{ weakly star in } L^\infty(0,T;\LpO),\\
 & &&\mbox{ weakly in } W^{1,2}(0,T;\Xzs'),\\
 w_k &\to w \quad && \hbox{ weakly in } L^2(0,T;\Xzs),\label{eq:convw1}\\
 \As w_k &\to \As w \quad && \hbox{ weakly in } L^2(0,T;\Xzs'),\label{eq:convAsw}\\
 \beta(u_k) &\to \bar \beta 
 \quad && \hbox{ weakly in } L^2(0,T;\LO),\label{eq:convbeta1}\\
 \Asigk u_k &\to \xi
 \quad && \hbox{ weakly in } L^2(0,T;\LO).\label{eq:convAsu}
\end{alignat}
%
%
%
%
It follows immediately that $\partial_t u + \As w = 0$ in $\Xzs'$.
Applying Lemma \ref{cor:weak_limit_s_zero-rev} with any $\beta > 0$ to
$\Asigk u_k(t)$ and $u_k(t)$ along
with the weak convergence relations \eqref{eq:convu1} and
\eqref{eq:convAsu}, we obtain $\xi(t) = u(t)$ in $\Xzb'$ for a.e.~$t \in
(0,T)$. Moreover, since $\xi(t)$ and $u(t)$ lie in $\LO$, which is dense
in $\Xzb'$, we see that $\xi = u$ a.e.~in $\Omega \times (0,T)$. 
For all $\phi \in \DO$ and $\varphi \in C^\infty_0(0,T)$, it follows
from \eqref{eq:cp_weak_k} that
\begin{align*}
\int^T_0 ( w_k(t) , \varphi ) \, \phi(t) \; \d t
&= \int^T_0 ( \Asigk u_k(t), \varphi ) \,
\phi(t) \; \d t
\\
&\quad + \int^T_0 \left( \beta(u_k(\cdot,t)) , \varphi \right) \phi(t) \; \d t - \int^T_0 (u_k(t),  \varphi) \,
 \phi(t) \; \d t.
\end{align*}
Passing to the limit as $k \nearrow \infty$, we obtain
\begin{align*}
\int^T_0 ( w(t) , \varphi ) \, \phi(t) \; \d t
&= \int^T_0 ( \xi(t), \varphi ) \,
\phi(t) \; \d t
\\
&\quad + \int^T_0 \left( \bar\beta(t), \varphi \right) \phi(t) \; \d t - \int^T_0 (u(t),  \varphi) \,
 \phi(t) \; \d t,
\end{align*}
which together with the density of $\DO$ in $\LO$ and the arbitrariness
of $\phi \in C^\infty_0(0,T)$ implies that
$$
w = \xi + \bar\beta - u \ \mbox{ in } \LO, \quad \mbox{
a.e.~in }(0,T).
$$
Thus we obtain $w(t) = \bar \beta(t)$ in $\LO$ for a.e.~$t \in (0,T)$,
or, in other words,
\begin{equation}\label{bb=w}
w = \bar \beta \quad \mbox{ a.e.~in } \Omega \times (0,T).
\end{equation}

For each $t \in [0,T]$, since $\{u_k(t)\}$ is bounded in $\LO$ and
$\LO$ is compactly embedded in $\Xzs'$, the sequence $\{
u_k(t)\}$ is precompact in $\Xzs'$. Moreover, $t \mapsto u_k(t)$ is
equicontinuous on $[0,T]$ with values in $\Xzs'$ for $k \in \mathbb
N$. Therefore, thanks to Ascoli's lemma, we infer that
\begin{equation}\label{eq:convu:g}
 u_k \to u \quad \hbox{ strongly in } C([0,T];\Xzs').
\end{equation}
Since $u_{0,k} \to u_0$ strongly in $\Xzs'$ by assumption, one
can check that $u(t) \to u_0$ strongly in $\Xzs'$ as $t \searrow
0$. In particular, $u(0) = u_0$.

Now, the major task is to identify the limit $\bar\beta$ as $\beta(u)$,
namely proving that $\bar\beta = \beta(u)$ a.e.~in $\RN\times
(0,T)$. To this end, we shall use Minty's trick,
i.e., we claim that
\begin{equation}\label{eq:minty}
  \limsup_{k\nearrow +\infty}\int_{0}^{T}\int_{\RN}\beta(u_k) u_k \, \dx \, \dt
  \le \int_{0}^T\int_{\RN}\bar\beta u \, \dx \, \dt.
\end{equation}
Actually, testing \eqref{eq:cp_weak_k} by $u_k$, we find that
\begin{equation}\label{eq:semicont1}
  \int_{0}^{T}\int_{\RN}\beta(u_k) u_k \, \dx\, \dt 
 = - \int_0^{T} \|u_k(t)\|_{\Xzsik}^2 \, \dt
 +\int_{0}^T \| u_k(t) \|^2_{L^2(\RN)} \, \dt 
   + \int_{0}^T\int_{\RN}w_k u_k \, \dx\, \dt.
\end{equation}
Taking the $\limsup_{k\nearrow +\infty}$ of both sides,
we have
\begin{align}\label{eq:semicont2}
  & \limsup_{k\nearrow +\infty}\int_{0}^{T}\int_{\RN}\beta(u_k) u_k \,\dx\,\dt\\
 \nonumber 
  & \le \limsup_{k\nearrow +\infty} \int_{0}^T 
      \left(\| u_k(t) \|^2_{L^2(\RN)} -
 \|u_k(t)\|_{\Xzsik}^2 \right) \, \dt
    + \int_{0}^{T} \int_{\RN} w u \, \dx \, \dt.
\end{align}
In particular, in order to take the limit of the last integral,
we used \eqref{eq:convw1} together with \eqref{eq:convu:g}, and observed
that
\begin{equation}\label{ni}
\int^T_0 \int_{\RN} w_k u_k \; \d x \, \d t
= \int^T_0 \langle u_k(t), w_k(t) \rangle_{\Xzs} \; \d t
\to \int^T_0 \langle u(t), w(t) \rangle_{\Xzs} \; \d t
= \int^T_0 \int_{\RN} w u \; \d x \, \d t.
\end{equation} 
The Poincar\'e inequality \eqref{eq:poincare-opt} gives
\begin{equation}\label{eq:Dk}
  D_k(t) := 
   \| u_k(t)\|^2_{L^2(\RN)} - \|u_k(t)\|_{\Xzsik}^2
   \le \left(\dfrac 1 {\lambda_1(\sigma_k)}-1\right) \|u_k(t)\|^2_{\Xzsik} \hbox{ for a.e. } t\in (0,T).
\end{equation}
Thus, recalling Proposition \ref{prop:asy_1eigen} and the energy
estimate \eqref{eq:energy1}, we conclude that
$$
  \limsup_{k\nearrow +\infty}\int_{0}^T D_k(t) \,\dt \le 0.
$$
Then \eqref{eq:minty} follows from the above along with \eqref{bb=w},
\eqref{eq:semicont2} and the fact that $u = 0$ outside $\Omega$.
Therefore thanks to \eqref{eq:convu1}, \eqref{eq:convbeta1}, and the
maximal monotonicity of the mapping $u \mapsto \beta(u)$ in $\LR \times
\LR$, one deduces that $\beta(u)=\bar \beta$ a.e.~in $\RN\times
(0,T)$. In particular, $\bar \beta = \beta(u)$ vanishes outside
$\Omega$. Hence by \eqref{bb=w} together with the fact that $w = 0$ in $\RN
\setminus \Omega$, one obtains
\begin{equation}\label{bb=w2}
\beta(u) = \bar \beta = w \quad \mbox{ a.e.~in } \RN \times (0,T),
\end{equation}
which also implies $\beta(u) = w \in L^2(0,T;\Xzs)$. 
Consequently, $u$ solves for almost any $t\in (0,T)$
\begin{equation*}
  \partial_t u + \As \beta(u) = 0 \ \hbox{ in } \Xzs'.
\end{equation*}
Note that, as a consequence of the procedure, recalling
\eqref{eq:convu1} again, we also get
\begin{equation}\label{u:forte}
  u_k \to u \quad \hbox{ strongly in } L^p(0,T;\LpO)
\end{equation}
by utilizing the uniform convexity of $\LpO$.

\subsection{Proof of Theorem \ref{th:fd}}\label{Ss:FD}
We first remark that, as in the Riesz representation theorem
for standard
Lebesgue spaces, one can also identify the dual space $(L^q_0(\RN))'$ of
$L^q_0(\RN)$ with $L^{q'}_0(\RN)$, where $q \in (1,\infty)$ and $q' :=
q/(q-1)$. 
%

Let $\sigma_k \searrow 0$ and let $(u_k,w_k)$ be the family of weak
solutions to
\begin{alignat}{4} \label{eq:ch_weak_k_mod}
 & \partial_t u_k + \As w_k = 0 \ &\hbox{ in }& \Xzs', \\
 & w_k = \Asigk u_k + \B(u_k) -\lambda_k u_k \ &\hbox{ in }& \Xzsik', \label{eq:cp_weak_k_mod}
\end{alignat}
with $u_k(0) = u_{0,k}$, where $\lambda_k := \lambda_1(\sigma_k)$
denotes the first eigenvalue of \eqref{eq:eigen} with $r$ replaced by $\sigma_k$.
As in \S \ref{Ss:UE}, (formally) test \eqref{eq:ch_weak_k_mod} by $w_k$
and \eqref{eq:cp_weak_k_mod} by $\partial_t u_k$ to get
$$
\|w_k(t)\|_{\Xzs}^2 + \dfrac{\d}{\d t} \Esk(u_k(t)) \leq 0 \quad \mbox{
for a.e. } t \in (0,T),
$$
where $\Esk : \Xzsik \to \mathbb R$ is defined as in
\eqref{eq:new_ene}. Indeed, the energy inequality above can be
rigorously derived as in
the proof of Theorem \ref{th:eu}. Integrate both sides over $(0,t)$. It
follows that
$$
\|w_k\|_{L^2(0,T;\Xzs)}^2 + \Esk(u_k(t)) \leq \Ezsit.
$$
By \eqref{unif:coerc}, we have
$$
\|w_k\|_{L^2(0,T;\Xzs)}^2 + \|u_k\|_{L^\infty(0,T;\LpO)}^p \leq
Q(\Ezsit),
$$
which also implies
$$
\|\partial_t u_k\|_{L^2(0,T;\Xzs')}^2 
+ \|\As w_k\|_{L^2(0,T;\Xzs')}^2
+ \|\beta(u_k)\|_{L^\infty(0,T;L^{p'}_0(\RN))}^{p'} 
\leq Q(\Ezsit).
$$
As in \eqref{est:Delsig}, by \eqref{eq:cp_weak_k_mod} and
estimates above along with the fact that $1 < p < 2$ (i.e., $p' > 2$),
we can take the unique bounded linear extension $\overline{\Asigk
u_k(t)} : L^{p'}_0(\RN) \to \mathbb R$ onto $L^{p'}_0(\RN)$ of the
functional $\Asigk u_k(t)|_{L^{p'}_0(\RN)} : \Xzsik \cap L^{p'}_0(\RN)
\to \mathbb R$ such that, by the identification $(\LpOd)' \simeq
\LpO$, 
\begin{align*}
\left\|\overline{\Asigk u_k(t)}\right\|_{\LpO}
 &\leq \|w_k(t)\|_{\LpO} + \lambda_k \|u_k(t)\|_{\LpO}
 + \|\beta(u_k(\cdot,t))\|_{\LpO}\\
 &\leq C \left( \|w_k(t)\|_{\LO} + \lambda_k
 \|u_k(t)\|_{\LpO} + \|u_k(t)\|_{\LpO}^{p-1}
\right)
\end{align*}
for some constant $C \geq 0$ independent of $k$ and $t$. 
We shall simply write $\Asigk u_k$ instead of $\overline{\Asigk
u_k(\cdot)}$ below. Thus we obtain
$$
\left\|\Asigk u_k\right\|_{L^2(0,T;\LpO)} \leq
Q(\Ezsit).
$$
Therefore, there exist weak limits $u$, $w$, $\bar \beta$ and $\xi$
such that, up to a (non-relabeled) subsequence,
\begin{alignat}{4}
 u_k &\to u \quad &&\mbox{ weakly star in } L^\infty(0,T;\LpO),
\label{FD:c:uLp:i}\\
 & &&\mbox{ weakly in } W^{1,2}(0,T;\Xzs'),\label{FD:c:duXs*:2}\\
 w_k &\to w \quad &&\mbox{ weakly in } L^2(0,T;\Xzs),\label{FD:c:wXs:2}\\
 \beta(u_k) &\to \bar \beta \quad &&\mbox{ weakly star in }
 L^\infty(0,T;\LpOd),\label{FD:c:betaLp':i}\\
 \Asigk u_k &\to \xi \quad &&\mbox{ weakly in }
 L^2(0,T;\LpO).\label{FD:c:DelsigLp:2}
\end{alignat}
Moreover, apply Ascoli's compactness lemma along with the compact embedding
$\LpO \simeq (\LpOd)' \hookrightarrow \Xzs'$ (see
\eqref{ass:comp_p_s}) to get
\begin{equation}\label{FD:c:uXsi*:C}
u_k \to u \quad \mbox{ strongly in } C([0,T];\Xzs').
\end{equation}
To prove $\xi = u$, we use Lemma
\ref{cor:weak_limit_s_zero-rev}. 
Indeed, choose $\beta$ sufficiently
large so that $\LpO$ is densely and continuously embedded in $\Xb' \
(\;\hookrightarrow \Xzb')$. Then the weak (star) convergence
of $u_k$ (cf.~\eqref{FD:c:uLp:i}) and that of $\Asigk u_k$ 
(cf.~\eqref{FD:c:DelsigLp:2}) suffice 
to apply the lemma (see (i) of Remark
\ref{Apdx:R:appli}) and obtain
the conclusion. As $\lambda_k \to 1$ by Prop.~\ref{prop:asy_1eigen},
we then arrive at
$$
\partial_t u + \As w = 0 \ \mbox{ in } \Xzs', \quad 
w = \bar \beta \ \mbox{ in } \Xzb'.
$$
Hence $w = \bar \beta$ in $\Omega \times (0,T)$.   
It remains to prove that $\bar \beta = \beta(u)$ in $\RN \times
(0,T)$. By Poincar\'e's inequality \eqref{eq:poincare-opt} and
\eqref{unif_bound-3}, we note that
$$
\int^T_0 \int_{\RN} \beta( u_k ) u_k \; \d x\, \d t
\leq \int^T_0 \int_{\RN} w_k u_k\; \d x\,\d t
\to \int^T_0 \int_{\RN} w u \; \d x \, \d t
= \int^T_0 \int_{\RN} \bar \beta u \; \d x \, \d t
$$
by \eqref{FD:c:wXs:2} and \eqref{FD:c:uXsi*:C}. Therefore by
\eqref{FD:c:uLp:i} and \eqref{FD:c:betaLp':i}, the maximal monotonicity
of $u \mapsto \beta(u)$ in $L^p(\RN) \times L^{p'}(\RN)$ yields $\bar
\beta = \beta(u)$ in $\RN \times (0,T)$. Moreover, $\beta(u)$ coincides
with $w$ on $\RN \times (0,T)$, as $w(\cdot,t)$ vanishes outside
$\Omega$. The rest of proof runs as in the proof of Theorem \ref{th:pm}.
%


\subsection{Proof of Theorem \ref{th:ac}}\label{Ss:AC}
Let $s_k \in (0,1)$, $s_k \searrow 0$ and let $(u_k,w_k)$ be the family of
weak solutions to the problem
\begin{alignat}{4} \label{eq:ch_weak_k_ac}
 & \partial_t u_k + \Askk w_k = 0 \ &\hbox{ in }& \Xzsk', \\
 & w_k = \Asi u_k + \B(u_k) - u_k \ &\hbox{ in }& \Esig', \label{eq:cp_weak_k_a}
\end{alignat}
with $u_k(0)=u_{0,k}$.
Compared to the proofs of Theorems \ref{th:pm} and \ref{th:fd},
this proof is definitely easier. Actually, since $\sigma$ is kept fixed,
the sequence $u_k$ retains some space compactness.

Put $X = \Esig$ and suppose that $s_k < \sigma$ for all $k \in \mathbb
N$ without any loss of generality.
Thanks to the uniform estimates \eqref{unif_bound-2}-\eqref{unif_bound-3}
obtained in \S \ref{Ss:UE} along with \eqref{hypo-ac}, one has, up to a (non-relabeled) subsequence,
\begin{alignat*}{4}
 u_k &\to u \quad &&\mbox{ weakly star in } L^\infty(0,T;\Esig),\\
 & &&\mbox{ weakly in } W^{1,2}(0,T;X'),\\
 \Asig u_k &\to \Asig u \quad &&\mbox{ weakly star in } L^\infty(0,T;\Xzsi'),\\
 w_k &\to w \quad &&\mbox{ weakly in } L^2(0,T;\LO),\\
 \Ask w_k &\to w \quad &&\mbox{ weakly in } L^2(0,T;X'),\\
 \beta(u_k) &\to \bar \beta \quad &&\mbox{ weakly in } L^2(0,T;\LO),
\end{alignat*}
which immediately gives $w = \Asig u + \bar \beta - u$ in $\Esig'$.
Here we used Lemma \ref{cor:weak_limit_s_zero-rev} with $\beta > 0$
sufficiently large so that $X' \hookrightarrow \Xzb'$ (as in \S \ref{Ss:PM})
to identify the limit of $\Ask w_k$ as $w$. 
By~\cite[Theorem 5]{Simon}, one can obtain
$$
u_k \to u \quad \mbox{ strongly in } C([0,T];\LO).
$$
By assumption, $u_{0,k} \to u_0$ strongly in $\LO$. Hence we obtain
$u(t) \to u_0$ strongly in $\LO$ as $t \searrow 0$. Moreover, by applying
Minty's trick to the maximal monotone operator $u \mapsto \beta (u(\cdot))$ in
$\LO$, one concludes that $\bar \beta = \beta(u)$. For all
$\varphi \in \DO \subset X$ and $\phi \in C^\infty_0(0,T)$, one can derive
$$
\int^T_0 \langle \partial_t u_k(t), \varphi \rangle_X \, \phi(t) \; \d t
+ \int^T_0 \langle \Ask w_k(t), \varphi \rangle_X \, \phi(t) \; \d t = 0.
$$
Passing to the limit as $k \nearrow \infty$, we obtain
$$
\int^T_0 \langle \partial_t u(t) + w(t), \varphi \rangle_X \, \phi(t)\; \d t = 0.
$$
Since $\DO$ is dense in $X$, we conclude that
$$
\partial_t u + w = 0 \ \mbox{ in } X', \quad 0 < t < T.
$$
Recalling that $X = \Esig$, we conclude that $u$ solves
$$
\partial_t u + \Asi u + \beta(u) - u = 0 \ \mbox{ in } \Esig', \quad 0 <
t < T.
$$


\section{Stationary states}
\label{sec:stat}

In this section we analyze the behavior of stationary states of system 
\eqref{eq:fCH}-\eqref{eq:bc} in the coercive case $p>2$.
We will put a particular emphasis on the asymptotic behavior of the 
stationary states when $\sigma\searrow 0$. 

\smallskip

The function $u$ is called a stationary state of \eqref{eq:fCH}-\eqref{eq:bc} 
when $\partial_t u = 0$ for a.e.~$(x,t)\in \Omega\times
(0,+\infty)$. Then by \eqref{eq:fCH} we have
$$
   (-\Delta)^s w = 0 \ \mbox{ in } \Omega, \quad w = 0 \ \mbox{ in }
  \mathbb R^N \setminus \Omega,
$$
which implies $w \equiv 0$ by Poincar\'e's inequality~\eqref{eq:poincare}.
Hence, by \eqref{eq:chem_pot}, $u$ solves the problem
\begin{equation}\label{eq:stat_state}
\begin{cases}
 (-\Delta)^\sigma u + \beta(u) - u = 0 \ &\mbox{ in } \Omega,\\
 u = 0 \ & \mbox{ in } \mathbb R^N \setminus \Omega.
 \end{cases}
\end{equation}
First of all, let us provide a weak formulation of
\eqref{eq:stat_state}. 
\begin{defi}\label{D:ss-weak}
 A function $u \in \Esig$ is called a \emph{weak solution} of
 \eqref{eq:stat_state}, if
\begin{equation}\label{eq:ss-weak}
\dfrac{C(\sigma)}2 \iint_{\RdN} \dfrac{\left(u(x)-u(y)\right)
 \left(\varphi(x)-\varphi(y)\right)}{|x-y|^{N+2\sigma}} \, \d x\, \d y
+ \int_\Omega \beta(u) \varphi \, \d x 
- \int_\Omega u \varphi \, \d x = 0
\end{equation}
for all $\varphi \in \Esig$. The weak form \eqref{eq:ss-weak} can be
 equivalently rewritten as
$$
\Asi u + B(u) - u = 0 \ \mbox{ in } \Esig'.
$$
\end{defi}
We next prove existence of a 
solution to \eqref{eq:stat_state}. To this end, we use the direct method
of calculus of variations. Recall that the energy functional
$ \mathbb{E}_{\sigma}: \mathcal E_\sigma \to \mathbb R$ 
is defined by
\begin{equation}\label{defienergy}
  \mathbb{E}_{\sigma}(v) := \dfrac{C(\sigma)} 4 \iint_{\mathbb R^{2N}}
       \dfrac{|v(x)-v(y)|^2}{|x-y|^{N+2\sigma}} \, \dx \, \dy
   + \int_\Omega \hat \beta(u) \, \dx 
   - \frac12 \int_\Omega |v|^2 \, \dx 
      \quad \mbox{ for } \ v \in \mathcal E_\sigma.
\end{equation}
It is easy to check that $\Esi$ is coercive in $\mathcal E_\sigma$. 
Indeed, by H\"older's and Young's inequalities,
$$
  \Esi(v) \geq \dfrac 1 2 \|v\|_{\Xzsi}^2 
   + \dfrac 1 {2p} \int_\Omega |v|^p \,\dx - C
    \quad \mbox{ for all } \ v \in \mathcal E_\sigma.
$$
Hence, the existence of a (global) minimizer $u\in \Esig $ follows 
from the compactness of the embedding
$\Xzsi \hookrightarrow \LO$. Note that $u$ actually solves
equation \eqref{eq:stat_state} in the sense of Definition
\ref{D:ss-weak}. Now, let $u$ be a global minimizer of $\Esi$.
 Since $\big||u(x)|-|u(y)|\big|\le |u(x)-
u(y)|$ for $(x,y)\in \RN\times\RN$, 
$\vert u\vert$ has the same
energy of $u$, namely $\Esi(u)=\Esi(|u|)$. Hence $|u|$ also minimizes $\Esi$
and solves \eqref{eq:stat_state}. Consequently, by applying maximum
principle for the fractional Laplacian (see,
e.g.,~\cite{ROS},~\cite{GrSer}) to the nonnegative solution $|u|$, we
infer that $|u| > 0$ or $u \equiv 0$ in $\Omega$. Therefore every
minimizer $u$ of $\Esi$ turns out to be sign-definite or identically
equal to zero over $\Omega$.
 Now, we are going to give conditions
implying that there exist nontrivial solutions to \eqref{eq:stat_state}. 
We can prove the following
\begin{propo}\label{prop:non_trivial}
 Let $\lambda_1(\sigma)$ be the first eigenvalue of \eqref{eq:eigen}
 with $r$ replaced by $\sigma$. Then,
 if $\lambda_1(\sigma) < 1$, problem \eqref{eq:stat_state} admits 
 a nontrivial weak solution.
\end{propo}
\begin{proof}
First of all, we claim that
$$
  \inf_{\mathcal E_\sigma} \Esi < 0 \quad \mbox{ if } \lambda_1(\sigma) < 1.
$$
Indeed, let $v_\sigma$ be the eigenfunction of $\Asig$ 
corresponding to the first eigenvalue $\lambda_1(\sigma)$,
normalized with respect to the $\LO$-norm. Then it is proved that
$\lambda_1(\sigma)$ is simple in~\cite[Prop.~9~(c)]{serva-valdi11}. Moreover,
 $v_\sigma$ is H\"older continuous up to the boundary by~\cite[Theorems
 1.1 and 1.3]{ROS-Poh} and~\cite[Proposition 4]{SeVa13}.
By a simple calculation, we have for $\epsi> 0$
\begin{align*}
 \Esi(\epsi v_\sigma)
  & = \dfrac{\epsi^2}2  \|v_\sigma\|_{\Xzsi}^2
   + \dfrac{\epsi^p}{p} \int_\Omega |v_\sigma|^p \, \dx 
   - \dfrac{\epsi^2}{2} \int_\Omega |v_\s|^2 \, \dx\\
 & = \dfrac{\epsi^2}2 \lambda_1(\sigma)
  + \dfrac{\epsi^p}{p} \int_\Omega |v_\s|^p \, \dx 
  - \dfrac{\epsi^2}{2} \\
 & = \epsi^2 \left[
 \dfrac 1 2 \left(\lambda_1(\sigma) - 1\right) + \dfrac{\epsi^{p-2}}p
  \int_\Omega|v_\s|^p \, \dx \right] < 0,
\end{align*}
provided that $\epsi$ is chosen so that
$$
  \dfrac 1 2 \left(\lambda_1(\s) - 1\right) 
   + \dfrac{\epsi^{p-2}}{p} \int_\Omega |v_\s|^p \, \dx < 0,
   \ \mbox{ i.e.~} \ 
   \epsi^{p-2} < \dfrac{p(1 - \lambda_1(\s))}{2 \int_\Omega |v_\s|^p
 \,\dx}.
$$
Hence the infimum of $\Esi$ over $\mathcal E_\s$ is negative whenever
$\lambda_1(\s) < 1$. Therefore every global minimizer is nontrivial.
\end{proof}
\begin{remark} {\rm
 If the first eigenvalue $\lambda_1$ of $-\Delta$ equipped with the homogeneous
 Dirichlet condition is not greater than one, then by the upper estimate
 $\lambda_1(\s) < \lambda_1^\s$ (see \eqref{eq:eigen_upper}), 
 we have $\lambda_1(\s) < 1$. Hence \eqref{eq:stat_state} possesses
 a nontrivial weak solution.
 }
\end{remark}
\begin{lemma}\label{L}
 Let $u$ be a weak solution of \eqref{eq:stat_state}. Then its energy is nonpositive, 
 i.e.~$\Esi(u)\le 0$. Moreover, if $\Esi(u) = 0$, then $u\equiv 0$.
\end{lemma}
\begin{proof}
Let $u$ be a weak solution of \eqref{eq:stat_state}. Test
 \eqref{eq:stat_state} by $u$ (i.e., substitute $\varphi = u$ in
 \eqref{eq:ss-weak}) to get
$$
\|u\|_{\Xzsi}^2  
   + \int_\Omega |u|^p \, \dx - \int_\Omega |u|^2 \, \dx 
  = 0,
$$
which yields
$$
  \Esi(u) = - \left( \dfrac 1 2 - \dfrac 1 p \right) \int_\Omega |u|^p \, \dx \leq 0.
$$
In particular, if $\Esi(u) = 0$, then $u \equiv 0$.
\end{proof}
As a consequence, we have the following criterion for non-existence 
of nontrivial solutions:
\begin{corollary}\label{C:nonex}
 If $\lambda_1(\s) \geq 1$, then \eqref{eq:stat_state} admits only
 the trivial solution.
\end{corollary}
\begin{proof}
 We observe that, by \eqref{eq:poincare-opt}, 
 \begin{align}
   \Esi(u) & = \dfrac 1 2 \|u\|_{\Xzsi}^2 
    + \dfrac 1 p \int_\Omega |u|^p \, \dx - \dfrac 1 2 \int_\Omega |u|^2 \, \dx
    \nonumber \\
   & \geq \dfrac 1 2 \left( \lambda_1(\s) - 1 \right) \int_\Omega |u|^2 \, \dx 
    + \dfrac 1 p \int_\Omega |u|^p \, \dx 
    \quad \mbox{ for all } \, u \in \mathcal E_\s.
 \label{est-b}
\end{align}
 Hence, if $\lambda_1(\s) \geq 1$, then $\Esi(u) \geq 0$. Therefore, due to Lemma
 \ref{L}, problem \eqref{eq:stat_state} has no nontrivial weak solution.
\end{proof}
\noindent%
We are now in position to state a result on the asymptotic behavior of nontrivial
weak solutions as $\s \searrow 0$.
\begin{propo}
 Suppose that $\lambda_1(\s)< 1 $ for all $\s \in (0,1)$.
 Let, for $\sigma\in(0,1)$, $u_\sigma$ 
 be a nontrivial weak solution of \eqref{eq:stat_state}.
 Then $u_\sigma$ converges to zero strongly 
 in $\LO$ as $\s \searrow 0$.
\end{propo}
\begin{proof}
Let $\sigma$ converge to $0$ along a sequence $\s_k$ and denote 
by $u_k$ the corresponding nontrivial weak solution to
 \eqref{eq:stat_state}. 
Let $r > 0$ and $v\in \mathcal E_{\s_k}$ be such that $\|v\|_{\LO} = r$. 
Let us first note that
$$ 
  \int_\Omega |v|^2 \, \dx 
   \leq |\Omega|^{(p-2)/p} \left( \int_\Omega |v|^p \, \dx
 \right)^{2/p}.
$$
Hence, combining this with \eqref{est-b}, we infer 
$$
  \Esi(v) \geq \dfrac {r^2} 2 (\lambda_1(\s_k) - 1) 
 + \dfrac{r^p}{p|\Omega|^{(p-2)/2}}
   = \dfrac{r^2}2 \left( \lambda_1(\s_k) - 1 +
 \dfrac{2r^{p-2}}{p|\Omega|^{(p-2)/2}} \right) \geq 0,
$$
provided that $\lambda_1(\s_k) - 1 + 2r^{p-2}/(p|\Omega|^{(p-2)/2})
 \geq 0$, which corresponds to~$r^{p-2} \geq (p/2) |\Omega|^{(p-2)/2} (1
 - \lambda_1(\s_k))$. 
Set
$$
 r_k := \left( \frac p 2 |\Omega|^{(p-2)/2} (1 - \lambda_1(\s_k))
 \right)^{1/(p-2)}.
$$
Then one has
$$
  \inf \big\{ \Esi(v) \colon v \in \mathcal E_{\s_k},    
     \ \|v\|_{\LO} \geq r_k \big\} \geq 0.
$$
Therefore, by Lemma~\ref{L},  
being $u_k$ non trivial, we have
$$
  \| u_k\|_{\LO} < r_k.
$$
Since $\lambda_1(\s_k) \to 1$ as $k\nearrow +\infty$ 
(cf.~Proposition~\ref{prop:asy_1eigen}), we
conclude that $u_k \to 0$ strongly in $\LO$ as $k\nearrow +\infty$.
\end{proof}

%
%


\section{Appendix}
 \label{sec:models}

Here we discuss in some more detail the relations
between our problem and other nonlocal Allen-Cahn 
or Cahn-Hilliard models. We just comment on the 
assumptions on the fractional diffusion operators,
neglecting the differences occurring in the choice
of the potential.

Let us start describing the relations between our problem
and the Cahn-Hilliard equation analyzed in \cite{BatesHan,CFG,GZ}.
Actually, these problems share a common variational 
structure as the variable $w$ is 
introduced as the first variation (in $L^2$) of some 
functional: in our case, of $\mathbb{E}_{\sigma}$
(cf.~\eqref{eq:def_energy}), while the nonlocal operator 
$J[u]$ of \cite{BatesHan,CFG,GZ} (recall~\eqref{eq:bates})
corresponds to the gradient of  
\begin{equation}\label{EJ}
  \mathbb{E}_J[u] = \frac14 \iint_{ \Omega \times \Omega }
   j(x-y) \big( u(x) - u(y) \big)^2 \, \dx \, \dy,
\end{equation}
as a direct computation shows.

Apart from the lower regularity of the kernel $K_r$ associated 
to $(-\Delta)^r$, one major difference is that the integration
domain in \eqref{EJ} is $\Omega \times \Omega$ in place of $\RR^{2N}$.
In a sense, this is similar to what happens in the case of the
so-called {\it regional Laplacian}\/ (cf., e.g., \cite{Guan}),
defined, for smooth functions, as (compare with~\eqref{def:fract_lapl})
\begin{equation}\label{reg:lapl}
  \Drreg u(x):= C(r,N) \pv \int_{\Omega} \frac{u(x)- u(y)}{\vert x-y\vert^{N + 2 r}} \,\dy.
\end{equation}
In other words, the above position corresponds to assuming 
that, as $x\in\Omega$, the value of $u(y)$ influences 
that of $u(x)$ only for $y\in\Omega$. Hence, in fact, no boundary 
conditions are taken in that case. In our setting, instead,
also the outer (Dirichlet) value $u(y)=0$, $y\in \RN\setminus \Omega$, 
carries some influence on the value of $\Dr u(x)$.

This difference is also reflected when one looks at the first
variation $J$ of $\mathbb{E}_J$ (cf.~\eqref{eq:bates}),
where the function $a(\cdot)$ depends in fact
on the variable $x\in \Omega$. 
Indeed, taking $W\equiv 0$ and neglecting
the principal value for simplicity,
also in our case it would be possible to write 
(at least formally)
$$ 
  (\delta \mathbb{E}_\sigma)(u) = a u - K_\sigma * u, 
   \quad a = \int_{\RN} K_\sigma (x - y) \, \dy.
$$
However, $a$, representing the ``total mass''
of $K_\sigma$, is now independent of $x$ (note, instead, that the 
convolution term is the same in both cases as we assume
$u$ be identically $0$ out of $\Omega$).

\smallskip

A further difference between the two models is related
to the regularity of the kernels. In \cite{BatesHan,CFG,GZ},
$j$ satisfies some kind of summability property.
On the other hand, in the case of the fractional Laplacian, 
the kernel $K_r$ is somehow ``less than $L^1$''. This is a trivial
remark, of course. However, some consequences from the point of view 
of regularity analysis deserve to be discussed. Indeed, in the
present case the term $\Ar u$ is {\it less regular}\/ than $u$
and this fact implies that the equation enjoys some {\it smoothing
effect}, as happens in the standard parabolic case
(i.e., for the usual Laplacian). This fact
permits us to ``embed'' compactness and density tools in 
the Hilbert formulation by working in
Hilbert triplets like $(\Xzr,\LO,\Xzr')$.

Instead, in the models analyzed 
in~\cite{BatesHan,CFG,GZ}, the term $J[u]$ is 
{\em strictly more regular}\/ than the function $u$ on
which $J$ acts. This means that the PDE system has limited 
regularization effects (actually, this is especially true for
Allen-Cahn based models, see, e.g.,~\cite{GS}; 
in the Cahn-Hilliard case the Laplacian
in the equation $u_t = \Delta w$ corresponding to 
our \eqref{eq:fCH} partially compensates
this). In other words, a singularity in the 
initial datum tends to propagate with time without smoothing 
out (at most it decreases in amplitude like in
a dissipative ODE). 

\smallskip

The regional Laplace operator \eqref{reg:lapl} characterizes also
the model studied more recently in \cite{ABG}. To be precise,
in \cite{ABG} the standard Laplacian is taken
in the analogue of \eqref{eq:fCH}, while
the regional operator $\Dsireg$ appears in the analogue of
\eqref{eq:chem_pot}. Moreover, no-flux conditions are
assumed for $w$. This, in particular, entails a mass-conservation 
property; namely, the spatial mean value of $u$ is constant
with respect to time, as one expects to occur 
in models describing phase separation,
on account of the underlying physics.
Using that $\Dsireg$ can be seen as a 
fractional power of the Neumann Laplacian
(restricted to the class of zero-mean functions), the authors
of~\cite{ABG} can show that, at least for sufficiently smooth
solutions, also the component 
$u$ turns out to satisfy a no-flux condition 
on $\partial\Omega$ (though no explicit boundary
condition is required in the mathematical formulation
of their problem).
It is also worth mentioning that the analysis given
in~\cite{ABG} admits a much wider class of potentials $W$,
including in particular singular functions like the ``logarithmic
potential'' mentioned in the Introduction. In principle,
this would be possible also for our model and we plan 
to address this issue in a forthcoming work. We 
also observe that, for the model considered in \cite{ABG},
one expects that the solution, at least asymptotically in
time and for non-singular functions $W$, could satisfy
strong regularization properties. This is due to the fact
that the regional Laplacian, as a fractional power 
of the Neumann Laplacian, satisfies the property 
$(-\Delta)^r_{\rm reg} \circ (-\Delta)^t_{\rm reg} 
= (-\Delta)^{r+t}_{\rm reg}$, which may allow use of 
bootstrap regularity methods. On the other hand, an analogue property
fails for our operators $\Ar$ due to the occurrence
of the ``solid'' Dirichlet condition. For this reason,
we expect that the analysis of asymptotic regularity properties
of weak solutions could be more challenging in our case.




\end{document}